\documentclass[11pt]{amsart}
\usepackage{amsmath,amssymb,mathtools}
\usepackage[T1]{fontenc}
\usepackage{lmodern,textcomp}
\usepackage{color}
\usepackage{graphicx}
\usepackage{amsthm}
\usepackage[all]{xy}
\usepackage{caption}
\usepackage{enumitem}
\usepackage{hyperref}
\usepackage{hhline}

\usepackage{geometry}
\usepackage{setspace}

\def\M{{\mathcal{M}}}
\def\V{{\mathcal{V}}}
\def\T{{\mathcal{T}}}
\def\A{{\mathcal{A}}}
\def\cS{\mathcal{S}}
\def\S{\mathcal{S}}
\def\C{{\mathcal{C}}}

\def\cC{{\mathcal{C}}}
\def\oC{\overline{\mathcal{C}}}
\def\oM{\overline{\mathcal{M}}}

\def\Om{{\Omega}}
\def\oOm{\overline{\Omega}}

\def\cO{\mathcal{O}}

\def\A{{\mathcal{A}}}
\def\cS{\mathcal{S}}

\def\RR{\mathbb{R}}
\def\CC{\mathbb{C}}
\def\ZZ{\mathbb{Z}}
\def\NN{\mathbb{N}}
\def\QQ{\mathbb{Q}}
\def\PP{\mathbb{P}}

\def\UU{\mathbb{U}}

\def\red{{\rm red}}

\def\res{{\rm res}}

\def\ab{{\rm ab}}
\def\nab{{\rm nab}}

\def\Stab{{\rm Stab}}
\def\Tw{{\rm Tw}}
\def\Star{{\rm Star}}
\def\tStar{{\rm tStar}}
\def\dStar{{\rm dStar}}
\def\Bic{{\rm Bic}}
\def\Dec{{\rm Dec}}

\def\sskip{\vspace{4pt}}

\def\oGamma{{\overline{\Gamma}}}

\def\red{}

\theoremstyle{definition}
\newtheorem{definition}{Definition}[section]
\newtheorem{notation}[definition]{Notation}

\newtheorem{remark}[definition]{Remark}

\theoremstyle{plain}
\newtheorem{conjecture}[definition]{Conjecture}
\newtheorem{theorem}[definition]{Theorem}
\newtheorem{problem}[definition]{Problem}

\newtheorem{proposition}[definition]{Proposition}

\newtheorem{lemma}[definition]{Lemma}
\newtheorem{corollary}[definition]{Corollary}

\begin{document}

\setlist[itemize,1]{leftmargin=18pt}
\setlist[enumerate,1]{leftmargin=18pt}

\title[Volumes of moduli spaces of flat surfaces]{Volumes of moduli spaces of flat surfaces}
\date{\today}
\author{Adrien Sauvaget}
\address{CNRS, Universit\'e de Toulouse, Institut de Mathématiques de Toulouse, UMR 5219, 118 route de Narbonne 31400 Toulouse, France}
\email{adrien.sauvaget@math.cnrs.fr}


\maketitle

\begin{abstract}  We study moduli spaces of flat surfaces with conical singularities of prescribed angles.  In the 90s Veech showed that these spaces are real-analytically diffeomorphic to moduli spaces of curves and endowed with a natural volume form.  We prove that the total volume of the space is finite.  When all angles are rational but not integral,  we show that Veech volumes can be explicitly computed and depend continuously on the angles.  This result is achieved in two steps:  we prove that Masur--Veech volumes of moduli spaces of $k$-differentials with no integral singularity are explicitly computable. Then, we study the asymptotic behavior of these volumes for large values of $k$.  

Finally, we study the extension of the volume function computed at strictly rational angles to angles data with an integral value.  We prove that a topological recursion relation holds at these special values of angles.  The proof of this result is based on an open conjectural expression of Masur--Veech volumes of moduli spaces of $k$-differentials with integral singularities.  Our result is a validation of this conjecture in the large $k$-limit. 
\end{abstract}

\setcounter{tocdepth}{1}
\setcounter{section}{-1}
\tableofcontents

\section{Introduction}\label{sec:intro}

 A {\em flat surface with conical singularities} (or {\em flat surface} for short in the text)  is a compact connected surface $C$ with a flat metric defined outside $n$ points $x_1,\ldots,x_n$ and such that: for all $i\in \{1,\ldots, n\}$ the neighborhood of $x_i$ is isomorphic to a cone with angle $2\pi a_i$ for some $a_i\in \RR_{>0}$. 
The genus of the surface is determined by the Gauss-Bonnet formula:
\begin{equation*}
2g(C)-2+n \, =\,  |a| \left(\overset{\rm def}{=}\sum_{i=1}^n a_i\right).
\end{equation*}
The vector $a=(a_1,\ldots,a_n)$ is called the {\em type} of the surface.  Two flat surfaces are isomorphic if there exists an isometry up to a constant scalar that preserves the ordering of the singularities.  In particular, two isomorphic surfaces share the same type.
 
 \sskip
 
Let $g$ and $n$ be non-negative integers  satisfying $2g-2+n>0$.  We denote by $\Delta_{g,n}\subset \RR^n$ the set of vectors $a$ such that $|a|=(2g-2+n)$, and by $\Delta_{g,n}^+=\Delta_{g,n}\cap\RR_{>0}^n$.  If we fix a vector $a\in \Delta_{g,n}^+$ then we denote by $\M(a)$ the {\em moduli space of flat surfaces of type $a$}.  There is a canonical isomorphism between $\M(a)$ and the moduli space of smooth curves of genus $g$ with $n$ markings
\begin{equation*}
\phi_a\colon \M(a)\overset{\sim}{\to} \M_{g,n}
\end{equation*}(this was proved by Thurston in genus 0~\cite{Thu},  and Troyanov in general~\cite{Tro}).

\subsection{Veech's measure}
 
 In~\cite{Vee1}, Veech produced systems of coordinates on $\M(a)$ in which the diffeomorphism $\phi_a$ is real-analytic.  He used these coordinates to describe the {\em holonomy foliation} whose leaves are locally defined as the level sets of the holonomy character map. This map assigns to a flat surface the {\em holonomy character}
  $$\pi_1(C\setminus \{x_1,\ldots,x_n\},\star )\to\mathbb{U}$$
computed by parallel transport (here $\mathbb{U}$ is the group of complex numbers of module 1).     If we fix a reference flat surface and a presentation of its fundamental group,  this assignment determines a map from a chart of $\M(a)$ to a  real torus $\mathbb{U}^{2g}$.  
 The holonomy foliation is smooth on a dense subset of $\M(a)$, and the leaves are endowed with a complex structure compatible with the complex structure of $\M_{g,n}$ via the morphism $\phi_a$.

 \sskip

Furthermore,   Veech constructed a  measure on $\M(a)$ by showing that the leaves of the holonomy foliation are canonically endowed with a volume form associated with a projective structure.   We consider the exterior product of this volume form with  the pull-back of the Haar measure of the torus along the holonomy character map. It defines a volume form  on an open dense subset of $\M(a)$ and thus a measure $\nu_{a}$ on $\M_{g,n}$ (see Section~\ref{ssec:forms} for the details and conventions).  We define the {\em flat volume function} as
\begin{eqnarray*}
{\rm Vol}\colon \Delta_{g,n}^+ &\to& \RR_{\geq 0}\cup \infty\\
a & \mapsto & \nu_{a}(\M_{g,n}).
\end{eqnarray*}

\begin{problem}\label{mproblem} Is the function ${\rm Vol}$ finite? Can we compute it? 
\end{problem}
McMullen solved this problem in genus 0 for $a\in ]0,1[^n$ using a generalization of the Gauss--Bonnet formula~\cite{McM}.  Koziarz--Nguyen proposed an alternative proof via the theory of {\red Kähler--Einstein} metrics~\cite{KozNgu}, and  Ghazouani--Pirio gave a heuristic solution for the case $(g,n)=(1,2)$ by computing the genus of the algebraic leaves of the holonomy foliation~\cite{GhaPir}.   Here, we solve Problem~\ref{mproblem} in all genera when $a$ is strictly rational. 
\begin{theorem}\label{th:veech}
If $n\geq 2$, then there exists a rational piece-wise polynomial $\mathcal{V}\colon\Delta_{g,n}^+ \to \RR$ of degree $4g-3+n$ such that the equality
\begin{eqnarray}
\label{eq:volexpresion} {\rm Vol}(a) &=& q(a)^{-1}\,  \V(a), \\
\label{eq:qexpresion} \text{ where }\,\,\, q(a)&=& \frac{(-1)^{g-1+n}}{4\, (2\pi)^{2g-2+n}} {{(2g-2+n)! \prod_{i=1}^n 2\, {\rm sin}(a_i\pi)}}
\end{eqnarray}
holds for all rational vectors $a$ without integral coordinates.
\end{theorem}
{\red The function $q(a)$ has poles along the hyperplanes  $\{a_i=N\}$ for $N\in \NN$.}   We will show that the piece-wise polynomial $\V$ is of class $C^1$ if $n\geq 3$ and $C^2$ if $n=2$, and vanishes at vectors with an integral coordinate (see Figure~\ref{fig:graph}).  These properties imply the following theorem. {\red
\begin{theorem}\label{th:veech1}
If $n\geq 2$, then the function $a\mapsto q(a)^{-1}\V(a)$ extends to a continuous function on $\Delta_{g,n}^+$ that we denote $\widetilde{\rm Vol}$. 
\end{theorem}
We do not establish the continuity of the function ${\rm Vol}$ here. However, by standard integration arguments, this function is lower semi-continuous (see Section~\ref{sssec:wallcrossing}), and Theorem~\ref{th:veech1} implies: 
\begin{corollary}\label{cor:veech1}
    For all $a\in \Delta_{g,n}^+$, the quantity ${\rm Vol}(a)$ is finite. If $n\geq 2$, then
     ${\rm Vol}(a)\leq \widetilde{\rm Vol}(a)$, and the equality holds for almost all vectors $a$. 
\end{corollary}}

We conjecture that the function ${\rm Vol}$ is continuous, as  observed in genus 0.  If this is true,  then equality~\eqref{eq:volexpresion} holds for all values of $a$.  To prove the continuity of Veech's volumes, one could construct a ``good'' compactification of $\M(a)$ and study the growth of Veech's volume form along the boundary.\footnote{The articles~\cite{BCGGM3,CosMoeZac} achieved the compactification of strata of differentials and the study of Veech's  form growth along the boundary.  However, extending these works to general families of flat surfaces is a delicate task as new topological phenomena emerge due to the lack of local residue conditions. Namely, a degenerate flat surface could contain a vertical node of twist 0, which is a singularity of infinite area leading to a contracted component. }

\begin{figure}[h]
\begin{center}
\includegraphics[scale=0.24]{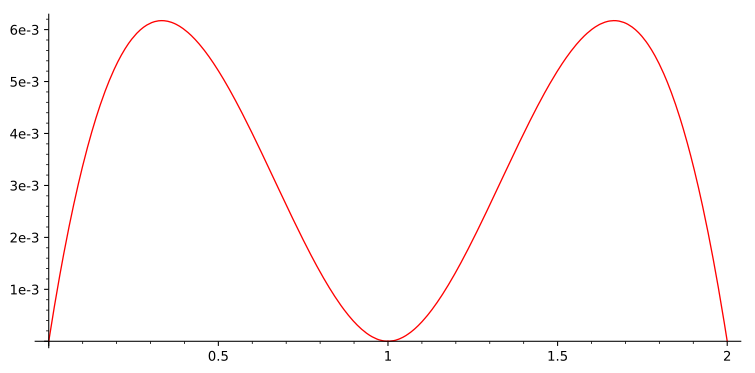} \hspace{10pt} \includegraphics[scale=0.24]{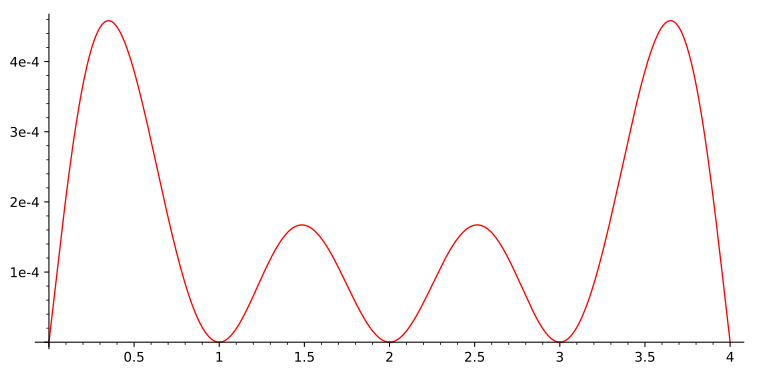}

\includegraphics[scale=0.24]{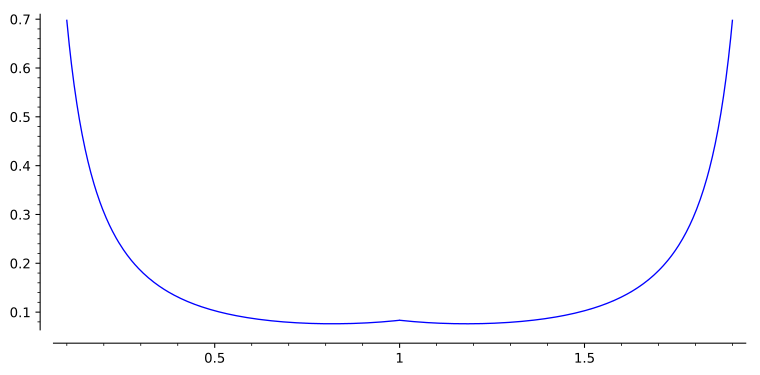} \hspace{10pt}  \includegraphics[scale=0.24]{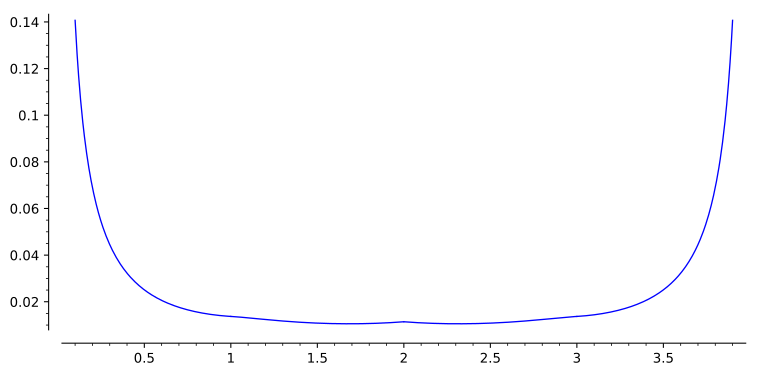}

\vspace{-5pt}

\end{center}
\captionsetup{width=0.80\linewidth}
\caption{\label{fig:graph} Graphs of   $a_1\mapsto (-1)^g\V(a_1,2g-a_1)$ (top), and  $a_1\mapsto {\rm Vol}(a_1,2g-a_1)$ (bottom), in genus  $g=1$ (left), and $g=2$ (right). }
\end{figure}

\subsection{Moduli spaces of pluri-canonical divisors} 

Let $a$ be a rational vector of ${\Delta}_{g,n}$.  We assume in this section that $a$ has at least one non-integral coordinate.    Let $k$ be  a positive integer such that $ka$ is integral.  A {\em $k$-canonical divisor of type $a$} is a smooth marked complex curve $(C,x_1,\ldots, x_n)$ such that the relation
\begin{equation}\label{eq:picrel}
\omega_{\rm log}^{\otimes k}\ = \ \mathcal{O}_C\left(ka_1\, x_1+\ldots+ka_n\, x_n\right)
\end{equation}
holds in the Picard group of $C$ (here $\omega_{\rm log}$ is $\omega_C(x_1+\ldots+x_n)$,  the cotangent line twisted by the set of markings).  We denote by $\M(a,k)$ the {\em moduli space of $k$-canonical divisors of type $a$}.  It is an algebraic sub-space of $\M_{g,n}$ of complex co-dimension $g$. 
 
\sskip

If $a$ is positive, then a $k$-canonical divisor carries a canonical flat metric of type $a$  with holonomy valued in $\mathbb{U}_k\subset \mathbb{U}$, the group of $k$th-roots of unity.  Conversely,  any flat surface of $\M(a)$ with holonomy valued in $\mathbb{U}_k$ is constructed from a $k$-canonical divisor, so we identify $\M(a,k)$ with a finite union of leaves of the holonomy foliation,  and endow $\M(a,k)$  with a canonical volume form.   The {\em Masur--Veech volume}  ${\rm Vol}(a,k)$ is the total volume of $\M(a,k)$ for this form, and it is finite (see~\cite{Vee}, \cite{Mas} for $k=2$, and \cite{Ngu} in general).  We prove the following Theorem along with Theorem~\ref{th:veech}.

\begin{theorem}\label{th:mv} If $a$ has no integral coordinate,  then ${\rm Vol}(a,k)$ can be explicitly computed.
\end{theorem} 

The computation of the different volume functions uses the line bundle $\mathcal{O}(-1)\to \M(k,a)$ locally generated by a meromorphic $k$-differential realizing the equality~\eqref{eq:picrel}. 
This line bundle carries a canonical hermitian metric: a non-vanishing section of $\mathcal{O}(-1)$ defines a family of flat surfaces, whose norm is the $k$-th power of the area of the fibers of this family (hence called the {\em area metric}).  We denote by $\omega(a,k)$ the curvature form associated with the dual hermitian metric and we set
\begin{equation}
\V(a,k)\coloneqq \int_{\M(a,k)} \omega(a,k)^{2g-3+n}.
\end{equation}
The proof of Theorem~\ref{th:veech}, and Theorem~\ref{th:mv} relies on the following strategy:
\begin{enumerate}
\item The volume ${\rm Vol}(a,k)$ is up to a simple factor equal to $\V(a,k)$  (Lemma~\ref{lem:measure1}).  
\item The integral $\V(a,k)$ can be expressed as the top interesection number of  $c_1(\mathcal{O}(-1))$ on an algebraic compactifiation of $\M(a,k)$ (see \cite{CosMoeZac} and Proposition~\ref{pr:goodarea}).
\item This intersection number can be computed explicitly by using the results on double ramification cycles of~\cite{BHPSS,CosSauSch}, and relations in the cohomology of moduli spaces of $k$-canonical divisors (Theorem~\ref{th:relpsi}).  
\end{enumerate}
The combination of these three points provides an explicit inductive procedure on $g$ and $n$ to compute the volumes ${\rm Vol}(a,k)$, thus proving Theorem~\ref{th:mv}.   The proof of Theorem~\ref{th:veech} is achieved by analyzing the asymptotic behavior of ${\rm Vol}(a,k)$ as $k$ goes to infinity (see Section~\ref{ssec:volumes}).   Indeed, if $a$ is rational,  then the Haar measure on $\mathbb{U}^{2g}$ is approximated by the counting measures of the subgroups $\mathbb{U}^{2g}_k$, which implies that
\begin{equation}
{\rm Vol}(a)=\underset{ka \in \ZZ^n}{\lim_{k\to \infty}} k^{-2g}\, {{\rm Vol}(a,k)}.
\end{equation}

\subsection{Flat recursion versus topological recursion} 
The functions $\V\colon\Delta_{g,n}^+\to \RR$ are determined by an inductive procedure on $g$ and $n$ that we call the {\em flat recursion} (FR).  
This name is chosen by analogy with the {\em topological recursion} (TR),    a theoretical frame that describes inductively the solutions to a broad class of enumerative problems, including  the  Weil--Petersson volumes of moduli spaces of hyperbolic surfaces with boundaries or conical singularities~\cite{Mir1,TanWonZha,DoNor}.  The FR does not fall in the TR formalism as it is defined by a larger class of graphs (see Section~\ref{ssec:FR}). However, a TR-type identity holds for vectors with an integral coordinate.

\begin{theorem}\label{th:tr} If $n\geq 3$, and $N$ is a positive integer then 
\begin{eqnarray}\label{for:tr}
N\frac{\partial\V}{\partial a_1}\bigg|_{a_1=N}  &=& \sum_{i=1}^n (a_i+N-1)\, \V(a_2,\ldots,a_{i-1},a_i+N-1,a_{i+1},\ldots)\\ \nonumber
&&+ \frac{1}{2} \int_{b=0}^{N-1} (N-1-b)b \,\,   \V(b,N-1-b,a_2,\ldots ) \, db\\ \nonumber
&& +\frac{1}{2}  \!\!\!\!  \sum_{\begin{smallmatrix} g_1+g_2=g \\  I_1\sqcup I_2=\{2,\ldots,n\} \end{smallmatrix}} b_1b_2 \, \V\left(b_1,  \{a_i\}_{i\in I_1}\right)\, \V\left(b_2, \{a_i\}_{i\in I_2}\right),
\end{eqnarray} 
where in the last sum we set $b_j=2g_j-2-\sum_{i\in I_j} (a_i-1)$ for $j=1$ or $2$.
\end{theorem}
Note that the LHS of~\eqref{for:tr} is well defined as  $\V$ is of class $C^1$. Furthermore, it determines the limit of the function ${\rm Vol}$ computed at strictly rational vectors along $a_1=N$.  Thus, this theorem states that Veech's volumes satisfy a TR-type relation at vectors with an integral coordinate.    

\sskip 

 We stress that this phenomenon is unexpected.  A TR-type identity typically reflects the compatibility of an enumerative problem with respect to the boundary of the moduli space of curves.  Theorem~\ref{th:tr} is proved by an intricate combinatorial argument, where  the only geometric input is the vanishing of the function $\V$ at vectors $a$ with integral coordinates.  This vanishing property is a consequence of the existence of isomonodromic deformations  and thus of an {\em isomonodromic foliation} on $\M(a)$ (transverse to the holonomy foliation).  The lack of a direct geometric argument to explain Theorem~\ref{th:tr} may be seen as a shadow of the transcendental nature of the isomonodromic foliation. 

\subsubsection*{A conjectural expression of Masur--Veech volumes} We propose a conjectural    expression of  Masur--Veech volumes as intersection numbers for general angle data (Conjecture~\ref{conj:int}, generalizing~\cite[Conjecture~1.1]{CheMoeSau}). If $a$ does not have integral coordinates, then ${\rm Vol}(a,k)$  is determined by $\V(a,k)$.  However,  if $a$ has integral coordinates then $\V(a,k)$ vanishes, so a general expression needs to involve other cohomology classes ($\psi$-classes here).  Theorem~\ref{th:tr} is a consequence of   Conjecture~\ref{conj:int} but our proof relies solely on the FR. Therefore, Theorem~\ref{th:tr} establishes the compatibility between the continuity of Veech's volume function 
and Conjecture~\ref{conj:int}. This is the first evidence towards this conjecture.

\subsection*{Acknowledgment} 

 This work stems from a question by Selim Ghazouani who presented the motivating problem and pushed me to delve into the seminal paper of Veech.  I would like to thank  Ga\"etan Borot,  Georges Comte, Dawei Chen, Bertrand Deroin, Alessandro Giachetto,  Elise Goujard, David Holmes, Martin Möller,  Gabriele Mondello,  Johannes Schmitt, Anton Zorich, and Dimitri Zvonkine for useful discussions. This list should be extended with the two anonymous reviewers who carefully read earlier versions of this work, including the most technical aspects, and contributed significantly to the present version. Finally, a thought goes to Nicolas Bergeron, both as an exciting teacher and a kind interlocutor on the early stages of these ideas. 

 \sskip 
 
 This research was partially supported by the Dutch Research Council (NWO) grant 613.001.651 and conducted during research visits at Leiden University.

\section{From intersection theory to volumes}\label{sec:vol}

This section contains most of the paper's geometric content. It starts with the description of moduli spaces of flat surfaces by Veech, and the definitions of compactifications of moduli spaces of $k$-canonical divisors.  Then, we express Veech and Masur--Veech volumes as intersection numbers, and present the flat recursion formula.  The main theorems are proven, assuming the technical cohomological results established in the next sections. 

\subsubsection*{Conventions} In all the text $|\star|$ stands for the cardinal if $\star$ is a set, or the size (i.e. the sum of the coordinates)  if $\star$ is a vector. If $E$ is a set of real vectors, then a {\em rational pair} (or {\em pair} for short) of $E$ is a couple $(a,k)$ where $a$ is a rational vector of $E$, and $k$ is a positive integer such that $ka$ is integral.

\subsection{Teichmüller space and the holonomy map}

We fix a reference oriented surface $\Sigma$ of genus $g$ with a set of markings $X=\{x_1,\ldots,x_n\}\subset \Sigma$. We denote 
$$
\widehat{\mathcal{T}}_{g,n} = \left\{\begin{matrix}\text{flat metrics on $\Sigma$ with cone }\\ \text{singularities at the markings}\end{matrix} \right\}\bigg/  \begin{matrix}\text{ isometries isotopic to identity}\\\text{ + dilatations}\end{matrix} 
$$ 
the {\em Teichmüller space of flat surfaces of genus $g$ with $n$ singularities}, and by $\widehat{\M}_{g,n}$  the quotient of $\widehat{\mathcal{T}}_{g,n}$ by the action of the mapping class group: it is the moduli space of all flat surfaces of genus $g$ with $n$ markings. For all $a\in \Delta_{g,n}^+$, we denote by $\mathcal{T}(a)$ the preimage of $\M(a)$ for the quotient map. 

\sskip

We fix a reference point $x$ of $\Sigma\setminus X$.  The holonomy character of a metric $\eta$  is the morphism $\chi_\eta\colon\pi_1(\Sigma\setminus X, x) \to \UU$ defined by assigning to any loop the angle of the rotation of the tangent space at $x$ defined by parallel transport of the metric along the loop.  To represent this function, we fix generators $\delta_1,\ldots,\delta_n, \alpha_1, \beta_1,\ldots,\alpha_g,\beta_g$ of $\pi_1(\Sigma\setminus X, x)$ such that $\delta_i$ is the class of a simple loop around  $x_i$, and the unique relation is
$$
\prod_{i=1}^n \delta_i= \prod_{j=1}^g [\alpha_j,\beta_j]. 
$$  
With this presentation of the fundamental group, the holonomy character of a flat metric is the data of a vector $\lambda=(\lambda_i)_{i=1}^{2g+n} \in K\subset \UU^{2g+n}$, {\red where $K$ is
the set} defined by $\prod_{i=1}^{n} \lambda_i=1$ (the $\lambda_i$ are the images of the reference loops). For all $a$ in $\Delta_{g,n}^+$ we denote by $K(a)$ the subset of $K$ defined by imposing $\lambda_i=e^{2i\pi a_i}$ for all $i\in\{1,\ldots,n\}$.  Then, the holonomy character map is a globally defined function 
$${\rm hol}\colon \widehat{\mathcal{T}}_{g,n} \to K$$
that restricts to a function ${\rm hol}\colon \T(a)\to K(a)$.

\subsection{Geodesic triangulations and parameters of $\widehat{\T}_{g,n}$} 

We fix a  flat metric {\red in} $\widehat{\mathcal{T}}_{g,n}$. The {\em Voronoi diagram} of the metric is the cell decomposition of $\Sigma$ defined as follows: the union of the $2$-,$1$-, and $0$-cells are respectively the sets of points $x$ of $\Sigma$ for which the equality  
\begin{equation*}
d(x, X) ={\rm length}(\gamma)
\end{equation*}
holds for 1, 2, and at least 3 geodesic paths $\gamma$ from $x$ to $X$ respectively  (here $d$ is the distance induced by the metric). 

\sskip

The Voronoi decomposition of $\Sigma$ admits a canonical dual cell decomposition, the {\em Delaunay decomposition}, defined as follows: if $x$ is 0-cell of the Voronoi decomposition, then we consider the embedded disk $D$ of center $x$ and radius $d(x,X)$, and the 2-cells of the Delaunay decomposition are the interiors of the convex envelope of the intersection of $\partial D$ with $X$. A {\em Delaunay triangulation} is a triangulation of $\Sigma$ obtained from the Delaunay decomposition by adding $1$-cells which are chords of the polygonal 2-cells. This triangulation is {\em geodesic}, i.e. the set of 0-cells is $X$ and the $1$-cells are geodesic paths. Delaunay triangulations will be needed further, but for now, we only use the fact that geodesic triangulations exist.

\sskip

 To describe the local structure of $\widehat{\mathcal{T}}_{g,n}$, we fix a geodesic triangulation (which does not have to be a Delaunay triangulation).  
By unfolding this triangulation, we can extract oriented geodesic paths $\gamma_1,\ldots, \gamma_{2g-1+n}$ such that the complement of  these paths is connected and simply connected. This data will be called a {\em polygonal representation} of the flat surface. The paths $\gamma_i$ form a basis of the relative homology group $H_1(\Sigma, X, \ZZ)$. Additionally, we can construct an alternative generating family of $\pi_1(\Sigma\setminus X)$ by considering the loops $\check{\gamma}_i$ crossing $\gamma_i$ and such that the tangent vectors to $\check{\gamma}_i$ and ${\gamma}_i$ form an orthogonal basis of the tangent space of the surface (see Figure~\ref{figure:torus}).    With this data, the space $H^1(\Sigma, X,\CC\times \UU)$ has parameters 
$$
z=(z_i)_{1\leq i\leq 2g-1+n}, e=(e_i)_{1\leq i \leq 2g-1+n}
$$  
in the dual basis of $(\gamma_i)_{1\leq i\leq 2g-1+n}$. We consider the open set $H\subset H^1(\Sigma, X,\CC\times \UU)$ defined by the constraints
\begin{equation*}
    \left\{\begin{array}{l} \sum_{i=1}^n (1-e_i)z_i = 0,\text{ and}\\
    z_i\neq 0 \text{ for all $i\in \{1,\ldots,2g-1+n\}$}.
    \end{array}\right.
\end{equation*}
It has a $\CC^*$-action defined by re-scaling the $z$-parameter and we denote by $\PP H$ the quotient. Our original choice of metric (up to a scalar) determines a point of $\PP H$, and a chart of $\widehat{\mathcal{T}}_{g,n}$ around this metric is an open set of $\PP H$. To define the transition functions between two charts, we remark that polygonal representations are related by a sequence of operations of two types:
 \begin{enumerate}
 \item {\em Reversing the orientation} of a geodesic $\gamma_i$. The transition map for this operation is $(z_i,e_i)\mapsto (-z_i,e_i^{-1})$. 
 \item {\em Cut-and-paste} for 2 consecutive edges of the polygon defined by geodesics $\gamma_i$ and $\gamma_{i+1}$ with the same orientation, and defining a chord $\gamma_i'$.  We construct another polygon by removing the triangle with sides $(\gamma_i, \gamma_{i+1},\gamma_{i}')$, and gluing this triangle to the rest of the polygon along the 2 edges defined by $\gamma_i$. The associated transition map is   $$(z_i,z_{i+1},e_i,e_{i+1})\mapsto (z_i+z_{i+1}, e_iz_{i+1}, e_i, e_{i+1}e_i^{-1}).$$
 \end{enumerate}

 \sskip 
 
This construction of the atlas of $\widehat{\mathcal{T}}_{g,n}$ allows for the construction of a $\CC^*$-bundle by gluing together open sets of the spaces $H$ and the maps  $H\to \PP H$ with the same rules. This defines a complex line bundle that is invariant under the action of the mapping class group,  thus a line bundle $\mathcal{L}\to \widehat{\M}_{g,n}$.

\begin{figure}
\includegraphics[scale=0.36]{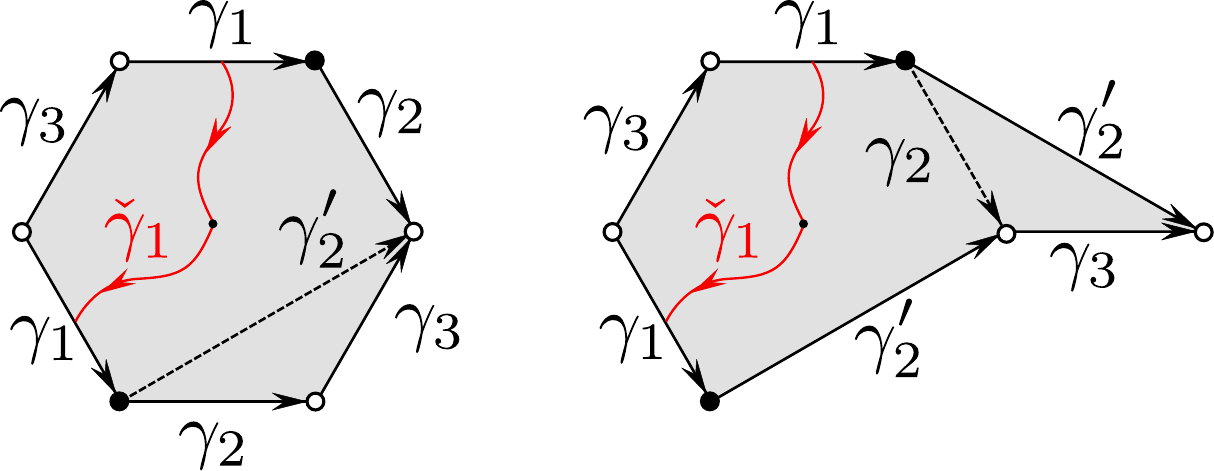}
\captionsetup{width=0.80\linewidth}
\caption{\label{figure:torus}Polygonal representation of a surface in $\M((2/3,4/3),6)$ related by a cut-and-paste operation. If $r=e^{2i\pi/6}$, then the coordinates of this surface in the left representation are $z=(-r,1,r)$ and $e=(r,-r,1)$. We also represented the dual path to $\gamma_1$ in red.}
\end{figure}

\sskip

With this differential structure, the holonomy character map is a continuous function:  the original choice of loops can be realized as a sequence of loops $\check{\gamma}_i^{\pm 1}$, and each $\lambda_j$ is a product of values $e_i^{\pm 1}$. Using this representation of the holonomy character map, Veech showed that ${\rm hol}$ is a submersion at a point $(z,e)$ if and only if $e\neq 1$~\cite[Theorem~0.3]{Vee1}. This defines a smooth foliation on $\T(a)$ outside of the locus of metrics with trivial holonomy. This foliation is invariant under the action of the mapping class group, thus defining a foliation on $\M(a)$. The leaves of the holonomy foliation have a complex structure for which the line bundle $\mathcal{L}$ is holomorphic. The smooth leaves are of complex dimension $2g-3+n$ while the singular ones are of dimension $2g-2+n$.  
 
\subsection{$\PP U(p,q,r)$-structure and volume forms}\label{ssec:forms}
Let $p, q,$ and $r$ be non-negative integers such that $p>0$.   Let $h$ be an hermitian quadratic form on $\CC^{p+q+r}$ of signature $(p,q,r)$, i.e. such that there exists a complex basis for which
\begin{equation}
h(x,y)=\sum_{i=1}^{p+q} h_i (x_i\overline{y}_i+\overline{x}_iy_i),
\end{equation}
with $h_i>0$  for $1\leq i\leq p$ and negative otherwise. The group $U(p,q,r)$ is defined as the group of matrices of size $(p+q)+r$ written in blocks as $
\begin{pmatrix}
A & B\\
0 & D
\end{pmatrix}$, with  $A$ in $U(p,q)$, $D$ diagonal, and $B$  any matrix of size $r\times(p+q)$.     We denote by $\mathcal{C}_h\subset \CC^{p+q+r}$ the positive cone, i.e. the set of vectors $x$ such that $h(x,x)>0$, and by ${\rm proj}\colon \cC_h\setminus\{0\} \to \PP \cC_h$ its projectivization.  We define 2 measures (in fact volume forms) on $\PP\cC_h$. The first one is 
\begin{equation}\label{def:nu1}
\nu_1(U) \coloneqq \text{Lebesgue measure}\left({\rm proj}^{-1}(U) \cap \{ x| h(x,x)\leq 1\} \right).
\end{equation}
The second is defined by considering the line bundle $\cO(-1)\to \PP_{\cC_h}$. It carries the hermitian metric equal to $h$ if we identify $\cO(-1)^*$ with $\cC_h^*$. We denote by $-\omega_h$ the curvature form of this hermitian metric. We define the second volume form as $\nu_2=\omega_h^{p+q+r-1}$. This second volume form is trivial if $r>0$ (while  the first one is not).

\begin{lemma}\label{lem:measure1}
We have $\nu_2=\frac{(p+q)! {\rm det}(h)}{\pi^{p+q}} \nu_1$.
\end{lemma}

\begin{proof} If $r>0$ then ${\rm det}(h)=0$ thus the lemma holds. For $r=0$, the proof is a generalization of the proof of Lemma~2.1 of~\cite{Sau4} and Lemma~2.1 of~\cite{CheMoeSau}. Using the action of the group $U(p+q)\cap U(p,q)$, we just need to compare the 2 forms on the set of vectors of the form $(x_1,0,\ldots,0,x_{p+1},0,\ldots)$.  We consider the chart of $\PP_{\cC_h}$ defined by $x_1=1$. In this chart, the measure $\nu_1$ is the measure associated with the differential form:
$$
\frac{2\pi }{h(x,x)^{p+q}{\rm dim}_\RR(\mathcal{C}_h)} \prod_{i=2}^{p+q}  (\frac{i}{2} dx_i\wedge d\overline{x}_i).
$$
In this same chart, the form $\omega_h$ is given by
$$
\omega_h=\frac{(h_1+h_{p+1} |x_{p+1}|^2) \, (\sum^{p+q}_{i=2} h_i dx_i\wedge d\overline{x}_i)- h^2_{p+1} |x_{p+1}|^2 dx_{p+1}\wedge d\overline{x}_{p+1}}{2i\pi(h_1+ h_{p+1} |x_{p+1}|^2)^2}.
$$
From this expression, we deduce the equality
\begin{eqnarray*}
\omega_h^{p+q-1}&=&\frac{(p+q-1)!\big(\prod_{i=1}^{p+q} h_i\big)}{(2i\pi)^{p+q-1} h(x,x)^{p+q}}\, \left(\prod_{i=2}^{p+q} dx_i\wedge d\overline{x}_i\right)= \frac{(p+q)! {\rm det}(h)}{\pi^{p+q}} \nu_2.
\end{eqnarray*}
\end{proof}

The line bundle $\mathcal{L}\to \widehat{\T}_{g,n}$ constructed in the previous section carries a hermitian metric $h$ defined as the area of the associated flat surface. We consider a chart of this bundle $U\to \PP U$  defined  by a polygonal decomposition. If $a$ is a vector of $\Delta_{g,n}^+$ and  $\lambda$ is a character in $K(a)$, then we consider the restriction of $\mathcal{L}^*$ to $\T(a,\lambda)={\rm hol}^{-1}(\lambda)$. We denote by $U(a,\lambda)$ the subset of $U$ obtained by intersection with the image of $\mathcal{L}^*|_{\T(a,\lambda)}$ under the chart map. The restriction of the area metric to $U(a,\lambda)$ is equal to a hermitian quadratic form $h_{a,\lambda}$ in the $z$-coordinates of $H^1(\Sigma, X, \CC)$ coming from the polygonal decomposition (and $U$ is contained in the positive cone for this quadratic form). The signature and the determinant of $h_{a,\lambda}$ only depend{\red s} on $a$, and we have
\begin{equation}\label{for:det}
{\rm det}(h_{a,\lambda}) \, =\, \frac{\pi^{2g-2+n}}{(2g-2+n)!\, } q(a),
\end{equation}
where the function $q(a)$ is defined by~\eqref{eq:qexpresion} in the introduction~\cite[Lemma 14.32]{Vee1}.  The transition maps between charts defined by cut-and-paste operations are  matrices in $U(p,q,r)$ that preserve the Lebesgue measure~\cite[Theorem~13.14]{Vee1}. Therefore, we  can construct a measure ${\nu}_{a,\lambda}$ on $\T(a,\lambda)$  locally defined by~\eqref{def:nu1}. If we denote by  $\omega(a,\lambda)$ the curvature form of the dual of the area metric, then Lemma~\ref{lem:measure1} asserts that
\begin{equation}\label{for:measure1}
\nu_{a,\lambda}= q(a) \, \omega(a,\lambda)^{2g-3+n}.
\end{equation}

We use the holonomy character map, to define Veech's measure on the total space $\mathcal{T}(a)$.  Outside the locus of trivial holonomy, the form $\nu_{a,\lambda}$ depends smoothly on the parameter $\lambda$, and  defines a form in {\red
$$
{\bigwedge}^{2(2g-3+n)}\left(\Omega^1_{\mathcal{T}(a)}\big/ {\rm hol}^* \Omega^1_{\mathbb{U}^{2g}} \right).$$}
If we denote by $\nu_{\mathbb{U}^{2g}}$ the Haar volume form, then
\begin{equation}
\nu_{a}\coloneqq {\rm hol}^{*} \nu_{\mathbb{U}^{2g}} \wedge \nu_{\alpha,\lambda}
\end{equation}
is a volume form on $\mathcal{T}(a)$ defined outside the locus of trivial holonomy which is of positive co-dimension in $\mathcal{T}(a)$. Therefore,  the form $\nu_a$ defines a measure on $\T(a)$.  This measure does not depend on the representation of the fundamental group of $\Sigma\setminus X$, and is invariant under the action of the mapping class group on $\mathcal{T}(a)$ (see~\cite[Lemma 13.6]{Vee1}). Therefore, it defines a measure $\nu_a$ on the moduli space of flat surfaces $\M(a)$.

\subsection{Three compactifications of $\M(a,k)$}\label{ssec:compactification}

Let $(a,k)$ be a rational pair of $\Delta_{g,n}$. The space $\M(a,k)$ is an algebraic sub-space of $\M_{g,n}$ ~\cite{Sch}. We denote by $\oM(a,k)$ the Zariski closure of $\M(a,k)$ inside the moduli space of stable curves $\oM_{g,n}$.  This space is called the {\em Deligne-Mumford compactification} of $\M(a,k)$.  A drawback of this first compactification is that the line bundle $\mathcal{L}$ considered in the previous section does not extend to the boundary. Yet, this space plays a central role here as it is used to define double ramification cycles (see Section~\ref{sec:DR}).

\sskip

 A {\em $k$-differential of type $a$} is the data of a  smooth marked  curve $(C,x_1,\ldots, x_n)$ and a meromorphic $k$-differential  $\eta$ such that ${\rm ord}_{x_i}\eta= ka_i -k$
for all $i\in \{1,\ldots,n\}$ {\red and it has no zeros and poles outside the marked points}. We denote by $\Omega(a,k)$ the {\em moduli space of $k$-differentials of type $a$}.  It is endowed with $\CC^*$-action and we denote by $\PP\Omega(a,k)$ its projectivization.  As a $k$-canonical divisor determines a unique $k$-differential up to a scalar, the space $\PP\Omega(a,k)$ is  isomorphic to $\M(a,k)$.   To define the second compactification of $\PP\Om(a,k)$ (or equivalently $\M(a,k)$),  we fix a vector of positive integers $P=(p_1,\ldots,p_n)$.  We denote by $\pi\colon\oC_{g,n}\to \oM_{g,n}$  the universal curve and by $\sigma_i\colon\oM_{g,n}\to \oC_{g,n}$ the section associated with the $i$-th marking for $i\in \{1,\ldots,n\}$.  For $P$ sufficiently large,  the space
\begin{equation}
V_P = \pi_*\omega^{\otimes k}_{\rm log}\left( \sum_{i=1}^n P_i\sigma_i\right),
\end{equation}
 is a vector bundle over $\oM_{g,n}$ and we have a canonical embedding of $\Om(a,k)$ in $V_P$.  We denote by $\PP\oOm(a,k)$ the closure of $\PP\Om(a,k)$ in $\PP V_P$.  This space is the {\em incidence variety compactification} of $\PP\Om(a,k)$ and  does not depend on the choice of $P$.  The line bundle $\mathcal{O}(1)$ is canonically isomorphic to $\mathcal{L}^{-\otimes k}$, and we denote $\xi=c_1(\mathcal{O}(1))\in H^2(\PP\oOm(a,k),\QQ)$.

 \sskip
 
 There is  a birational morphism $\PP\oOm(a,k)\to \oM(a,k)$ defined by forgetting the differential (which is not an isomorphism in general) but  both spaces are singular.  A third compactification called the {\em moduli space of multi-scale differentials}, has been constructed in~\cite{BCGGM3,CosMoeZac}.   It is a de-singularization of the incidence variety compactification $\PP\Xi\oM(a,k)\to \PP\oOm(a,k)$. We will not recall the precise description of this space here,  although it plays an important role in the present work as it was needed to establish the following comparison result.
 
 \begin{proposition}[Theorem~1.4 of \cite{CosMoeZac}]\label{pr:goodarea} If $a$ is positive, then for all $d\geq 0$ the cohomology class of the current associated to $\omega(a,k)^d$  is equal to $\xi^d$ in $H^{2d}(\PP\oOm(a,k),\QQ)$. 
 \end{proposition}
 This proposition together with~\eqref{for:measure1} provides the following expression of Masur--Veech volumes. \begin{proposition}\label{pr:volint} The integral $\V(a,k)$ vanishes unless  the vector $a$ has no integral coordinates. In this case, we have
 \begin{equation}\label{for:volint}
     q(a){\rm Vol}(a,k)\ =\  \int_{\PP\Om(a,k)} \left(k^{-1}\omega(a,k)\right)^{2g-3+n} \ =\ \int_{\PP\oOm(a,k)} \left(k^{-1}\xi\right)^{2g-3+n}.
 \end{equation}
 \end{proposition}
 
 \subsection{Flat recursion formula}\label{ssec:FR}
 
 Here, we construct the functions $\V\colon \Delta_{g,n}^+ \to \RR$ of Theorem~\ref{th:veech}  by induction on $g$ and $n$. {\red In order to present this inductive procedure, we first need to introduce some combinatorial structures.

 \subsubsection{Star graphs and twists}\label{sssec:stargraphs}
 
\begin{definition} A {\em stable graph} of genus $g$ with $n$ legs is the datum of $$\Gamma=(V,H,g:H\to \NN,
\iota:H\to H,\phi:H\to V,H^\iota\simeq \{1,\ldots, n\}), \text{ where: }$$
\begin{enumerate}
\item $V$ is a finite set of {\em vertices};
\item $H$ is a finite set of {\em half-edges};
\item the function $\iota$ is an involution of $H$: the set $E$ of cycles of length $2$ for $\iota$ is the set of {\em edges}, while the fixed points of $\iota$ are called {\em legs}, and we fix an identification of the legs with $\{1,\ldots,n\}$;
\item if $v$ is a vertex, then the  {\em valency} $n(v)$  is the cardinal of $\phi^{-1}\{v\}$, and it has to satisfy $$2g(v)-2+n(v)>0;$$ 
\item we have the equality $g=h^1(\Gamma)+\sum_{v\in V} g(v)$, where $h^1(\Gamma)=|E|-|V|+1;$
\item the graph is connected.
\end{enumerate}
\end{definition}

\begin{definition}
A {\em star graph} is a pair $(\Gamma,v_0)$, where $\Gamma$ is a stable graph and $v_0$ is a vertex of $\Gamma$  (the {\em central vertex}) satisfying: all edges are between $v_0$ and a vertex different from $v_0$ (an {\em outer} vertex).
\end{definition}

\begin{definition}
A {\em twist} on a stable graph $\Gamma$ is a function $b:H\to \RR$ satisfying:
\begin{itemize}
\item for all $v\in V$, we have $\sum_{h\in \phi^{-1}(v)} b(h) = 2g(v)-2+n(v);$
\item   if $(h,h')$ is an edge of $\Gamma$, then  $b(h)+b(h')=0$;
\item if $(h_1,h_1')$ and $(h_2,h_2')$ are edges between the same vertices $v$ and $v'$, then $b(h_1)\geq 0$ if and only if $b(h_2)\geq 0$; in which case we denote $v\geq v'$;
\item the relation $\geq $ defines a partial order on the set of vertices. 
\end{itemize}
The pair $(\Gamma,b)$ is called a {\em twisted graph}.  If $e=(h,h')$ is an edge of $\Gamma$,  then its twist $b(e)$ is the absolute value of $b(h)$ or $b(h')$.  The  {\em multiplicity} of a twisted graph is
\begin{equation}
m(\Gamma,b)=\prod_{e \in {\rm Edges}} b(e). 
\end{equation}
If $v$ is a vertex, then we denote by $b(v)= (b(h))_{h\in\phi^{-1}(v)}$, i.e. the vector in $\RR^{\phi^{-1}(v)}$ of values of the twists at the half-edges incident to the vertex $v$. Finally, if $a$ is a vector in $\Delta_{g,n}$, we say that a twist $b$ is {\em compatible with $a$} if $b(i)=a_i$ for all $i\in \{1,\ldots,n\}$. 
\end{definition}

\begin{definition}
A twisted graph $(\Gamma,b)$ is a {\em  twisted star graph} if there exists a (unique) vertex $v_0$ such that $(\Gamma,v_0)$ is a star graph, and $v_0$ is smaller than the outer vertices.
\end{definition}

\begin{notation}
    We denote by $\Stab_{g,n}, \Star_{g,n}, {\rm Tw_{g,n}},$ and $\tStar_{g,n}$ the sets  of stable, star, twisted, and twisted star graphs of genus $g$ with $n$ legs. We denote by $\Star^1_{g,n}$ the set of star-graphs such that the first leg is incident to the central vertex.

 If $(\Gamma,v_0)$ is a star graph in ${\rm Star}_{g,n}^1$, and $a$ is a vector in $\Delta_{g,n}^+$, then we denote by $\Delta_{\Gamma}(a)$ the set of twists on $\Gamma$, compatible with $a$, and that make $\Gamma$ a twisted star graph with $v_0$ as central vertex.\footnote{We do not include $v_0$ in the notation as the central vertex is the vertex carrying the first leg.}
\end{notation}

An important observation is that an element $\Delta_{\Gamma}(a)$ is fully determined by the values of the $b(e)$ (as the orientation of the edge is prescribed by the condition of being a star graph). Therefore, $\Delta_{\Gamma}(a)$ is the simplex in $\RR_{>0}^{E(\Gamma)}$ of dimension $h^1(\Gamma)$ defined by a system of equations indexed by outer vertices
$$\left\{\sum_{e\mapsto v} b(e) + \sum_{i\mapsto v} a_i = 2g(v)-2+n(v)\right\}_{v\in V\setminus\{v_0\}}$$ 
where the sums are over edges and legs incident to $v$.

\begin{figure}
$$\xymatrix@=1.5em{
2\ar@{-}[d]
&& 3 \ar@{-}[d]   \\ 
*+[Fo]{3}\ar@/_/@{-}[rd]_{e_1} \ar@/^/@{-}[rd]^{e_2}   && *+[Fo]{2} \ar@{-}[ld]^{e_3}   \\ 
&*+[Fo]{1}, \ar@{-}[ld]\\
1
}$$
\captionsetup{width=0.80\linewidth}
\caption{ \label{ex:star} Example of star graph in ${\rm Star}_{7,3}^1$: vertices are the circle with the genera indicated inside, while we indicated labels for the legs and edges (not half-edges). On this example the simplex $\Delta_{\Gamma}(a)$ is the set of positive numbers  $(b(e_1),b(e_2),b(e_3))$ satisfying $b(e_1)+b(e_2)=7-a_2$ and $b(e_3)=4-a_3$. It is empty if $a_2\geq 7$ or $a_3\geq 4$. }
\end{figure}

 \subsubsection{Construction of $\V$ by the flat recursion} For the base of the induction we set  $\V\left(\Delta_{0,3}^+\right)=1$, and $\V(\Delta_{g,1}^+)=0$ for all $g\geq 1$. If $2g-2+n>1$ and $n>1$, then we set 
\begin{eqnarray}
\label{for:FR} \V(a)&\coloneqq &   \sum_{\Gamma \in {\rm Star}^1_{g,n}} \frac{1}{|{\rm Aut(\Gamma)|}} \V_\Gamma(a), \text{ where} \\
\label{for:FRgamma}
\V_\Gamma(a)&\coloneqq &   
 \!\!\! \int_{b \in \Delta_\Gamma(a)}  m(\Gamma,b)\cdot \mathcal{A}\left(b(v_0)\right) \cdot \left( \prod_{v\neq v_0} \V\left(b(v_0)\right)\right) \, db.     
\end{eqnarray}
In this expression, $v_0$ is the central vertex and the functions $\A$ is the polynomial defined by 
\begin{eqnarray}\label{def:A}
\mathcal{A}\colon\Delta_{g,n}&\to& \RR \\
\nonumber a &\mapsto&  [z^{2g}] \,\, {\rm exp}\!\left(\frac{a_1z\cS'(z)}{\cS(z)}\right) \frac{(-a_1)^{2g-3+n} }{\cS(z)^{2g-1+n}}\prod_{i=2}^n \cS(a_iz),
\end{eqnarray}
where $\cS(z)=\frac{{\rm sinh}(z/2)}{z/2}$, and the notation $[z^{2g}]$ stands for the coefficient of $z^{2g}$ in the formal series. Note that this polynomial is only symmetric in the $n-1$ last variables, so we impose that $a_1$ is  the first entry of the vector $b(v_0)$.  The presence of these functions in the recursive definition of the function $\V$,  accounts for the large $k$ limit of integrals of $\psi$-classes on moduli spaces of $k$-canonical divisors (see~\cite{CosSauSch} and Section~\ref{sec:FR}).
}

\begin{theorem}\label{th:fr}
The intersection numbers $\V(a,k)$ can be   computed explicitly for all rational pairs $(a,k)$ of $\Delta_{g,n}^+$. Moreover, there exists a positive constant $C$ such that, for all pairs we have
\begin{equation}
\left| k^{-4g+3-n}\V(a,k) - \V(a)\right| < Ck^{-1}.
\end{equation}
\end{theorem}

\subsection{Computation of Masur--Veech and Veech volumes}\label{ssec:volumes}

We explain how to use Theorem~\ref{th:fr}   to complete the proof of Theorems~\ref{th:veech},~\ref{th:veech1}, and~\ref{th:mv}. First, let $a$ be a rational vector of $\Delta_{g,n}^+$ without integral coordinate, and let $k$ be a positive integer such that $ka$ is integral.  We have ${\rm Vol}(a,k) = q(a)^{-1}k^{-2g+3-n} \V(a,k)$ by Proposition~\ref{pr:volint}. The first part of Theorem~\ref{th:fr} implies that $\V(a,k)$ is explicitly computable and so is ${\rm Vol}(a,k)$. This completes the proof of Theorem~\ref{th:mv}.

\subsubsection{Convergence for large $k$-values} To complete the proof of Theorem~\ref{th:veech}, we consider the sequence of measures on $\M_{g,n}$ defined by $\nu_{a,k}(U)=\nu_{a,k}(U\cap \M(a,k))$ for all open sets $U\subset \M_{g,n}$ (here $\nu_{a,k}$ is the Masur--Veech measure).  The  measure $\nu_a$ is the weak limit of the measure $k^{-2g}\nu_{a,k}$ as $k$ goes to infinity: indeed  the uniform probability measures of $K(a,k)$ (the subset of $K(a)$ of vectors with coordinates in $\UU_k$) weakly converge to the Haar measure on $K(a)$. Then, the following inequality holds
\begin{eqnarray}
{\rm Vol}(a)& \leq&\underset{ka \in \ZZ^n}{\lim_{k\to \infty}} {k^{-2g}} {\rm Vol}(a,k) \\
\nonumber &=& q(a)^{-1}\V(a) < +\infty.
\end{eqnarray}
The equality between the first and second lines is the consequence of Theorem~\ref{th:fr}.  Theorem~\ref{th:veech} will be established by showing that the first inequality is an equality. To do so, we recall that the space $\M(a)$ may be partitioned into $\bigsqcup_D \M_D$ where $D$ runs over the  classes of decompositions up to isotopy.  These classes can be encoded as maps on surfaces as in~\cite{Ngu} or~\cite{LanZvo}, and there are finitely many such combinatorial objects. 
The space $\M_D$ associated with a class of decompositions $D$ is the locus of flat surfaces with Delaunay decomposition in the class $D$. It is of null measure unless $D$ is  a triangulation so  we have
$$
\nu_a(\M_{g,n})= \sum_{\begin{smallmatrix}\text{$D$ class }\\\text{of triangulations} \end{smallmatrix}} \nu_a(\M_D). 
$$
We fix a class of triangulation $D$ and we will show  that 
$$\nu_a(\M_D)=\underset{ka \in \ZZ^n}{\lim_{k\to \infty}}   k^{-2g}\nu_{a,k}(\M_D).$$
We choose a polygonal decomposition associated with the Delaunay triangulation. It determines local parameters $(z_i,e_i)_{1\leq i\leq 2g-1+n}$ of $\widehat{\T}_{g,n}$. We denote by  $K_D(a)$ the subset of $\UU^{2g-1+n}$ defined by imposing that the cone angles of the flat surfaces are prescribed by $a$: the monodromy around a marking is a product of coordinates $e_i^{\pm 1}$ so this constraint is a polynomial equation. We denote by $K_D(a,k)$ the intersection of $K_D(a)$ with $\UU_k^{2g-1+n}$. 

\sskip

 We recall that a Delaunay decomposition is dual to a Voronoi diagram. The Delaunay decomposition is a triangulation if and only if each vertex of the Voronoi diagram is trivalent.  We denote by $H_D$ the subset of  $\CC^{2g-1+n}\times K_D(a)$ defined by \begin{equation*}
    \left\{\begin{array}{l} \sum_{i=1}^n (1-e_i)z_i = 0,\text{ and}\\
    {\rm length}(\gamma) > 0 \text{ for all edges $\gamma$ of the Voronoi diagram.}
    \end{array}\right.
\end{equation*}
The lengths of the edges of the Voronoi Diagram are rational functions in the parameters $z_i$ and $e_i$, thus  $H_D$ is a subset of an affine space defined by a finite set of algebraic equalities and inequalities. We denote by $U_D\subset H_D$ the set of flat surfaces of area at most one (this constraint is again an algebraic inequality). The space $\M_D$ is the quotient of $\PP H_D$ by the automorphism group of $D$ so we have
$$
\nu_a(\M_D)= \frac{1}{\big|{\rm Aut}(D)\big|}\text{Lebesgue measure $\times$ Haar measure}(U_D).
$$
\begin{proof}[{\red Proof of Theorem~\ref{th:veech}}] We denote by $f\colon K_D(a)\to \RR_{\geq 0} \cup \{\infty\}$ the function defined by integration of the Lebesgue measure in the fiber of a point via the holonomy foliation.  There exists a finite partition of $K_D(a)$ into subanalytic sets such that the function $f$ is either infinite or a polynomial in finitely many subanalytic functions of the parameters of $K_D(a)$ and their logarithms  ~\cite[Theorem 1]{ComLioRol}. For each subset $E\subset K_D(a)$ where the function $f$ is continuous and finite
$$\int_E f=\underset{ka \in \ZZ^n}{\lim_{k\to \infty}}   \frac{1}{\left| K_D(a,k)\cap E\right|} \sum_{e\in  K_D(a,k)\cap E} f(e).$$
Moreover, if a subset $E$ where $f$ is infinite had a non-zero measure then $E$ would have at least one point in $K_D(a,k)$ for $k$ sufficiently large which would contradict the finiteness of Masur--Veech volumes. Altogether we get
$$
\nu_a(\M_D) =  \frac{1}{\big|{\rm Aut}(D)\big|} \int_{K_D(a)} f = \underset{ka \in \ZZ^n}{\lim_{k\to \infty}}   \frac{1}{\left| K_D(a,k) \right|} \sum_{e\in  K_D(a,k)} f(e) = \underset{ka \in \ZZ^n}{\lim_{k\to \infty}}   \nu_{a,k}(\M_D).
$$
\end{proof}

\subsubsection{Wall-crossing properties of the flat recursion}\label{sssec:wallcrossing}  The function $\V$ is a continuous piece-wise polynomial of degree at most $4g-3+n$ on $\Delta_{g,n}^+$. This can be proved by induction on $g$ and $n$ (with base $(g,n)=(0,3)$).  Indeed, if  $\Gamma$ is a star graph, then the contribution of $\V_\Gamma$ to the flat recursion relation~\eqref{for:FRgamma} is the integral of a piece-wise polynomial on a simplex whose boundaries depend linearly on the parameter $a$. As the function integrated is continuous on the closure of the domain of integration (which is compact), this integral is continuous. The degree of $\V_{\Gamma}$ is at most {\red
\begin{equation*}
\underset{\text{degree of $\mathcal{A}$}}{\underbrace{4g(v_0) - 3+n(v_0) }} + \!\!\underset{\text{edge factors}}{\underbrace{|E|}} \!\!+ \sum_{v\neq v_0}\underset{\text{degree of $\mathcal{\V}$}}{\underbrace{4g(v) - 3+n(v)}}  + \!\!\underset{\text{dimension of $\Delta_\Gamma$ }}{\underbrace{h^1(\Gamma) }} \!\!\leq 4g-3+n.
\end{equation*}}
 We have a better characterization of the regularity of $\V$ at vectors with integral coordinates.
\begin{lemma}\label{lem:regularity}
If $2g-2+n>1$ and $n\geq 2$, then the following properties hold:
\begin{enumerate}
\item if $(g,n)\neq (0,3)$, then $\V(0,\ldots,0,2g-2+n)=0$; 
\item for all $i\in\{1,\ldots,n\}$ and $N$ positive integer, there exists a continuous piece-wise polynomial $\widetilde{\V}$ such that
$$
\V\, =\, \left\{ \begin{array}{cl} (a_i-N) \widetilde{\V} & \text{ if $n\geq3$},\\
(a_i-N)^2 \widetilde{\V} & \text{ if $n=2$} \end{array} \right. 
$$  
in a neighborhood of the line $a_i=N$.
\end{enumerate}
\end{lemma}

\begin{proof}
To prove the first point, we observe that the contribution of a star graph to the flat recursion formula is divisible by $a_1$ unless $g_0=0$ and $e+n_0=3$ (because the $\mathcal{A}$-term is divisible by $a_1^{2g_0-3+e+n_0}$). However, in this situation, the domain $\Delta_\Gamma$ is empty so the associated contribution is trivial. 

\sskip

For the second point, we first remark that the rational numbers $\V(a,k)$ are invariant under permutation of the coordinates of $a$ so the same holds for $\V$ after passing to the large $k$ limit by Theorem~\ref{th:fr}. So we can choose $i=2$ without loss of generality.  Besides, {\red as $\V(a,k)=q(a) {\rm Vol}(a,k)$, the rational number $\V(a,k)$ is of the sign of $q(a)$}, and the same holds for $\V(a)$ by passing again to large $k$ limit. In particular, $\V$ vanishes at vectors $a$ with an integral coordinate. 
\footnote{We do not know how to prove these properties of the function $\V$ using only its definition~\eqref{for:FR}.}

\sskip

The walls of the chambers of polynomiality of $\V$ are defined by equations of the form $\sum_{i\in I} a_i = N$ for a positive integer $N$ and a subset $I$ of $\{1,\ldots,n\}$. We choose a point of the wall defined by $\{a_2=N\}$ that does not sit on another wall. Around this point, we have
$$
\V=P + \mathbf{1}_{a_2< N}\,  Q,
$$
where $P$ and $Q$ are polynomials, and {\red $\mathbf{1}_{a_2< N}$} is the function equal to 1 for $a_2<N$ and $0$ otherwise.  We will prove that $Q$ vanishes with {\red multiplicity $2$} along $\{a_2=N\}$ if $n\geq 3$, and with multiplicity $3$ if $n=2$. The lemma follows from this property: the polynomial $P$ vanishes along $\{a_2=N\}$ because $\V$ does; if $n=2$, then $P$ vanishes with multiplicity at least $2$ because $\V$ is of constant sign.

\sskip

We denote by $S\subset {\rm Star}_{g,n}^1$, {\red the set of star graphs such that $2$ is incident to some outer vertex $v_2$ and no other leg is incident to this vertex.} The polynomial $Q$ is the sum of the contributions defined by~\eqref{for:FRgamma} for graphs in $S$. If $\Gamma$ is a graph in $S$ then the associated contribution is of the form: {\red
\begin{equation*}\label{for:FRgamma1}
\V_\Gamma=\int_{\sum_{e\mapsto v_2} b(e)=a_2-N} \left(\prod_{e \mapsto v_2} b_{e}\right)  Q_\Gamma(a,(b(e))_{e\mapsto v_2}) \,\, db_{j},
\end{equation*}
where $Q_\Gamma$ is a polynomial in the variables $a$ and the twists at the edges incident to $v_2$.}
This shape is obtained by performing the integration over the parameters {\red $b({e})$ for $e$ non-incident to $v_2$ }in~\eqref{for:FRgamma}. This polynomial vanishes with multiplicity at least $2$ along $\{a_2=N\}$. Therefore, we denote by $S_1\subset S$  the set of graphs such that the vertex with the second leg has 1 edge. Moreover, we denote by $Q_1$ and $Q_{>1}$ the sum of contributions of graphs in $S_1$ and $S\setminus S_1$ respectively.  We have $Q=Q_1+Q_{>1}$, and $Q_{>1}$ vanishes with multiplicity at least $3$ along $\{a_2=N\}$, so it remains to study $Q_1$.

\sskip 

The polynomial $Q_1$ is of the form 
$$
(a_2-N)\, \V(a_2, a_2-N)\,  \widetilde{Q}_1(a_2-N,\{a_i\}_{i\neq 2}),
$$
{\red as the twist at the unique edge of a graph in $S_1$ is given by $a_2-N$}. Therefore $Q_1$ vanishes with multiplicity at least 2 by the first point of the lemma, and the second point of the lemma holds when $n\geq 3$. If $n=2$, then we have the equality:
$$\widetilde{Q}_1\bigg|_{a_1=2g-N,a_2=N}=\V(2g-N,0).$$ 
 Therefore $\widetilde{Q}_1$ vanishes along $\{a_2=N\}$ by the first point of the lemma, and $Q_1$ vanishes with multiplicity at least 3. 
\end{proof}

{\red
\begin{proof}[Proof of Theorem~\ref{th:veech1}] 
    As the equality ${\rm Vol}(a)=\V(a)/q(a)$ holds at strictly rational numbers, it can be extended to a continuous function outside the loci defined by $a_i=N$. Besides, Lemma~\ref{lem:regularity} implies that the limit of ${\rm Vol}$ along the hyperplane $a_i=N$ is well-defined and finite. So ${\rm Vol}$ can be extended to a continuous function on all of $\Delta_{g,n}^+$.
\end{proof}

\begin{proof}[Proof of Corollary~\ref{cor:veech1}] First, we assume that $n\geq 2$. The function ${\rm Vol}$ is a lower semi-continuous on $\Delta_{g,n}^+$. Indeed, if we fix $\epsilon>0$, then there exists a compact $K\subset \M_{g,n}$ such that $\nu_a(K)>\nu_a(\M_{g,n})-\epsilon/2$. As  $\nu_a(K)$ depends continuously of $a$, there exists an open set $U\subset \Delta_{g,n}^+$, containing $a$, and such that $\nu_{a'}(K)>\nu_a(k)-\epsilon/2$ for all $a'\in U$ so that 
$$
\nu_{a'}(\M_{g,n})> \nu_{a'}(K) > \nu_a(\M_{g,n})- \epsilon.
$$
Therefore, ${\rm Vol}$ is a lower semi-continuous function that is equal to a continuous function $\widetilde{\rm Vol}$ on a dense set of points of $\Delta_{g,n}^+$. As a result,  ${\rm Vol}\leq \widetilde{\rm Vol}<+\infty$ on $\Delta_{g,n}^+$,  and 
$$\{{a \in \Delta_{g,n}^+ \text{ s.t. } \rm Vol}(a)< \widetilde{\rm Vol}(a)\} = \bigcup_{n>0} \{a \in \Delta_{g,n}^+ \text{ s.t. } {\rm Vol}(a)\leq  \widetilde{\rm Vol}(a)-1/n\}.$$ Each of the sets in the right-hand side is of measure $0$, so the union is of measure 0 and  the identity ${\rm Vol}(a)=\widetilde{\rm Vol}(a)$ holds for almost all values of $a$.
\sskip

Finally, we observe that we have a canonical morphism $\M(a_1,\ldots,a_n,1)\to \M(a_1,\ldots,a_n)$ (the universal curve), and the MV measure $\M(a_1,\ldots,a_n,1)$ is the product of the area form in the direction of the fibers and the MV measure of the base.  Therefore after integrating along the coordinate in the fibers, we obtain the identity ${\rm Vol}(a_1,\ldots,a_n,1)={\rm Vol}(a_1,\ldots,a_n).$ This implies that the finiteness of the Veech volume also holds for $n=1$ as ${\rm Vol}(2g-1)={\rm Vol}(2g-1,1)$ which is finite.
\end{proof}}

 \subsection{Flat surfaces with integral singularities}\label{ssec:tr} 
 
 We have seen that Proposition~\eqref{pr:volint} cannot be used to compute Masur--Veech volumes for vectors $a$ with integral coordinates. The following conjecture provides a more general expression of Masur--Veech volumes as intersection numbers.
 \begin{conjecture}\label{conj:int}
 Let $m\geq 0$, and let $(a,k)$ be a rational pair $\Delta_{g,n}^+$ such that $a_i$ is integral for all $i\leq  m,$ and strictly rational otherwise. If $k\geq 2$, then we have
\begin{eqnarray}
\label{eq:volexpresionINT} {\rm Vol}(a,k) &=& \frac{1}{k^{2g-3+n}q_m(a)} \, \int_{\PP\oOm(a,k)} \xi^{2g-3+n-m} \prod_{i=m+1}^n \psi_i, \\
\label{eq:qexpresionINT} \text{ where }\,\,\, q_m(a)&=& \frac{(-1)^{g-1+n+\sum_{i=1}^m a_i}}{4\, (2\pi)^{2g-2+n-m}}{{(2g-2+n)! \prod_{i=m+1}^n 2\, {\rm sin}(a_i\pi)}},
\end{eqnarray}
and the class $\psi_i\in H^2(\oM_{g,n},\QQ)$ is the Chern class of the cotangent line at the $i$-th marking. 
 \end{conjecture}
 If this conjecture holds, then all Masur--Veech volumes can be explicitly computed. Furthermore, this conjecture predicts the value of the volume of $\M(a)$ for all rational vectors $a$. For simplicity, we restrict to the case $m=1$. For all $i\in \{1,\ldots,n\}$, we denote  \begin{equation}
 \V^i(a,k)\coloneqq \int_{\PP\oOm} \psi_i\,  \xi^{2g-4+n}.
 \end{equation}
  The following lemma is proved in the same way as Theorem~\ref{th:fr}.
  \begin{lemma}\label{lem:defvpsi} Let $N$ be a positive integer. There exists a piece-wise polynomial $\V^i\colon \Delta_{g,n}^+ \to \RR$ and a positive constant $C$ such that 
\begin{equation}\Big| k^{-4g+4-n} \V^i(a,k)-\V^i(a)\Big| < C\, k^{-1}
\end{equation}
for all pairs $(a,k)$ with $a_i=N$. Futhermore,  $N\, \V^1|_{a_1=N}$ is equal to the RHS of~\eqref{for:tr}.
 \end{lemma}
 Theorem~\ref{th:tr} will be a direct corollary of the following theorem.
\begin{theorem}\label{th:psider}
If $N$ is a positive integer, then 
\begin{equation}\label{for:psider}
\left( \frac{\partial \V}{\partial a_i}-\V^i \right)\bigg |_{a_i=N}=0.
\end{equation}
\end{theorem}
 If $a$ is a rational vector such that $a_1$ is integral while the other entries are strictly rational,  then Theorem~\ref{th:psider} implies that
\begin{equation}
\underset{a'\to a}{\lim_{a'\in (\QQ\setminus\ZZ)^n}} {\rm Vol}(a') = q_1(a)\,  \V^1(a).
\end{equation}
The LHS of this equality is equal to ${\rm Vol}(a)$ if we assume that the volume function is continuous, while the RHS is equal to ${\rm Vol}(a)$ if we assume that Conjecture~\ref{conj:int} holds.\footnote{Similar results  hold for general value of $m$   and can be deduced from the case $m=1$. We do not include these results here as a general study of the combinatorics of Conjecture~\ref{conj:int} will be done in subsequent work with Chen and Möller.}

\section{Intersection theory of moduli spaces of $k$-differentials}\label{sec:local}

In this section, we review the results of~\cite{BCGGM2} on the incidence variety compactification.   Moreover, we establish a set of relations satisfied by cohomology classes of boundary components (Theorem~\ref{th:relpsi}).  These relations generalize the results of~\cite{Sau} (in the case  $k=1$)  and will be used in the following sections to explicitly compute  the intersection numbers involved in the expression of the various volume functions.

\subsection{Canonical cover and residue conditions}\label{ssec:res}
Let $(a,k)$ be a rational pair of $\Delta_{g,n}$.  If we choose a $k$-differential $(C,x_1\ldots,x_n,\eta)$  in $\Omega(a,k)$, then  we define 
 $$
\red  \widehat{C}\coloneqq \text{smooth completion of } \left\{ (x,v)\in T_{C\setminus \{x_1,\ldots,x_n\}}^\vee, \text{ such that } v^{k} =\eta\right\}. 
$$
The map $f\colon\widehat{C}\to C$ is a cyclic ramified cover of degree $k$.  The curve $\widehat{C}$ carries a canonical differential $v$ such that $v^{k}=f^*\eta$. Each point $x_i$ with singularity of order $m$ has ${\rm gcd}(m,k)$ preimages under $f$ with  ramification order $k/{\rm gcd}(m,k)$. Besides, the order of $v$ at each point is determined by $a$.  Therefore,  a pair $(a,k)$  determines a triple $(\widehat{g}, \widehat{n},\widehat{a})$ such that we have an embedding 
$$
\Omega(a,k)\hookrightarrow \Omega(\widehat{a},1) \Big/ \mathbb{U}_k,
$$
where the $\mathbb{U}_k$-action is defined by permuting the labels of preimages of  a singularity.  This morphism is the {\em canonical cover morphism}.

\sskip


Let $i$ be an element in $\{1,\ldots,n\}$ such that $a_i$ is a negative integer, and let $(C,x_1,\ldots,x_n,\eta)$ be a differential in $\Omega(a,k)$.  The point  $x_i$ has $k$ preimages under the canonical cover map.  These preimages are poles of order $a_i$ of the differential on the covering curve,  and the residues at two such points  differ by a $k$-th root of unity.  The {\em residue} at $x_i$ is defined as the $k^{\rm th}$ power of any of these residues. 
We denote by $\res_i\colon \Omega(a,k)\to \CC$ the {\em $i$th residue morphism}, i.e.  the morphism defined by mapping $\eta$ {\red to its residue at $x_i$}. 
 
 \sskip
 
If $R\subset \{1,\ldots,n\}$ is a set of indices $i$ such that $a_i\in \ZZ_{<0}$,  then the moduli space of $k$-differentials {\em with residue conditions} $\Om^R(a,k)$ is the sub-space of $\Om(a,k)$ of differentials with vanishing residue at $x_i$ for $i\in R$.  We denote by $\PP\Om^R(a,k)$ its projectivization and by $\PP\oOm^R(a,k)$ the closure of $\PP\Om^R(a,k)$ in $\PP\oOm(a,k)$ (still called the incidence variety compactification).  If $i$ is an element of $\{1,\ldots,n\}\setminus R$ such that $a_i\in \ZZ_{<0}$, then the morphism $\res_{i}$ is a section of the line bundle $\mathcal{O}(1)\to \PP\Om^R(a,k)$ that extends to the boundary of the incidence variety compactification. 

\begin{lemma}\label{lem:vanres} The section $\res_{i}$ vanishes with multiplicity $k$ along $\PP\Om^{R\cup\{i\}}(a,k)$.
\end{lemma}

\begin{proof} 
If $k=1$, then the residue morphism is a submersion, thus the vanishing multiplicity of $\res_{i}$ along $\PP\Om^{R\cup\{i\}}(a,k)$ is $1$ (see Corollary~3.8 of~\cite{Sau}). For higher values of $k$, we use the canonical cover morphism to embed $\Omega(a,k)\hookrightarrow \Omega(\widehat{a},1) \big/ \mathbb{U}_k$. 
Then the residue at $x_i$ is the $k$-th power of the residue at any of its marked preimages of the canonical cover. The residue morphism is a submersion along the image of $\Omega(a,k)$ in $\Omega(\widehat{a},1) \big/ \mathbb{U}_k$. Therefore,  the residue morphism at $x_i$ vanishes with multiplicity $k$. 
\end{proof}

\subsection{$k$-decorated graphs}\label{ssec:dec} 

{\red In the present text we need several families of graphs to describe the different compactifications of moduli spaces of $k$-canonical divisors. We summarize their interplay in the following diagram}
\begin{equation}
\xymatrix@R=1.6em@C=0.6em{
\Dec \ar[d] &\supset & \Bic &\supset & {\begin{matrix} \dStar, \\ 2-\Bic \end{matrix}}  \ar[lld] & & \\
\Tw \ar[d]&\supset &\tStar \ar[d] & \supset & {\rm \red sStar}&& \\
\Stab & & \ar[ll] \Star&&&&
}
\end{equation}
(the arrows are maps defined by forgetting part of the data defining a class of graphs).
{\red In the following table recalls their section of definition and the sections in which they will be manipulated. 
\begin{center}
\begin{tabular}{|c|c|c|}
\hline
Class of graph & Definition & Used in\\ 
\hhline{|=|=|=|}
$\Stab, \Star, \Tw,$ and $\tStar$  & Section~\ref{sssec:stargraphs} & all sections  \\
\hline
$\Dec$   & Section~\ref{ssec:dec} &  only used to define ${\rm Bic}$ \\
\hline
$\Bic$  & Section~\ref{ssec:dec} & Section~\ref{sec:local}, and~\ref{sec:FR} \\
\hline
$\dStar$  & Section~\ref{ssec:dec} & Section~\ref{sec:FR} \\
\hline
$2-{\rm Bic}$  & Section~\ref{ssec:dec} & Section~\ref{sec:DR} \\
\hline
${\rm sStar}$  & Section~\ref{ssec:DR}  & Section~\ref{sec:DR} \\
\hline
\end{tabular}
\end{center}}

\begin{definition}
A {\em decorated graph} is the datum of 
$$
\overline{\Gamma}=(\Gamma,b,\ell, V(\Gamma)=V^{\rm ab}\sqcup V^{\nab}), \text{ where:}
$$

\begin{itemize}
    \item $(\Gamma,b)$ is a twisted graph;
    \item $\ell:V(\Gamma)\to \{0,-1,\ldots,-d\}$ is a level function, i.e.  surjective function such that for all vertices $v$ and $v'$, $(v\leq v')$ implies that $\ell(v)\leq \ell(v')$; 
    \item  all twists at half-edges adjacent to vertices in $V^{\rm ab}$ are integral.
\end{itemize} 
The integer $d$ is called the {\em depth} of the graph. We denote by {\rm Dec} the set of decorated graphs. 
\end{definition}

\begin{definition}
A {\em bi-colored graph} is a decorated graph of depth $1$ satisfying:  all edges are between a  vertex of level $0$  and a vertex of level $-1$.  We denote by $\Bic$ the set of bi-colored graphs.   Two special classes of bi-colored graphs will be considered here:
\begin{enumerate}
\item $2-\Bic$: the set of graphs with $2$ vertices such that the vertex of level 0 is in $V^{\rm ab}$. This class of bi-colored graphs is used in Section~\ref{sec:DR}.
\item $\dStar$: the set of graphs with a single (central) vertex of level $-1$ in $V^{\rm nab}$,  each vertex in $V^{\rm ab}$ has a single edge to the central vertex.  This class  is used in Section~\ref{sec:FR}.
\end{enumerate}
\end{definition}

\begin{figure}
\includegraphics[scale=1.2]{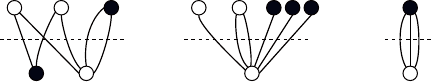}

(a) \hspace{80pt} (b) \hspace{80pt} (c)
\captionsetup{width=0.80\linewidth}
\caption{ \label{ex:star} {\red Examples of decorated graphs in (a) ${\rm Bic}$, (b) ${\rm dStar}$, and (c) $2-{\rm Bic}$. We have only indicated the levels (the vertices of level 0 are above the dotted line), and the set of vertices in $V^{\rm ab}$ (in black).}  }
\end{figure}

\begin{definition} Let $a$ be a vector in $ \Delta_{g,n}$ and $k\in \ZZ_{>0}$.   We say that a twist function $b$ is: 
\begin{itemize}
\item {\em compatible with $a$} if $b(i)=a_i$ for all $i \in \{1,\ldots, n\}$ {\red (this was already defined in Section~\ref{sssec:stargraphs})};
\item a {\em $k$-twist} if $kb$ has integral values.  
\end{itemize}
\end{definition}

\begin{notation}
If ${\rm X}$ is one of the sets of twisted or decorated graphs,   then we denote by ${\rm X}(a)$  and  ${\rm X}(a,k)$  the subsets of twisted/$k$-twisted graphs compatible with $a$.
\end{notation}

\subsection{Strata associated to bi-colored graphs}\label{ssec:strata}

{\red We recall that $\oM_{g,n}$ admits a canonical stratification indexed by stable graphs: the stratum associated with a stable graph $\Gamma$ is the locus of curves with dual graph $\Gamma$. The closure of this stratum is the image of the morphism
\begin{equation}
 \zeta_{\Gamma}\colon \oM_\Gamma \left(\coloneqq \prod_{v\in V(\Gamma)}\oM_{g(v),n(v)} \right) \to \oM_{g,n},   
\end{equation}
which is a finite and of degree $\big|{\rm Aut}(\Gamma)\big|$.}

\sskip

We consider a moduli space $\PP\Om^R(a,k)$ of $k$-differentials with residue conditions  as in Section~\ref{ssec:res}.  The incidence variety compactification of this moduli space admits a stratification indexed by decorated graphs of $\Dec(a,k)$. The stratum associated with a graph $\overline{\Gamma}$ is of co-dimension at least $1+d+h,$ where $h$ is the number of horizontal edges (edges with a zero twist) and $d$ is the depth of $\overline{\Gamma}$ (see~\cite[Section 6]{BCGGM2}).  Therefore, a stratum is of co-dimension at least 2 unless $\overline{\Gamma}$ is either:
\begin{itemize}
 \item a decorated graph of depth 0 with 1 horizontal edge,
 \item or a bi-colored graph.
\end{itemize}

We are only interested in relations between cohomology classes of strata of co-dimension 1 in $\PP\Om^R(a,k)$, so  we  will only recall the description of the strata associated with bi-colored graphs (we will see that graphs of depth 0  do not contribute to these relations).  Let $\overline{\Gamma}$ be a bi-colored graph in $\Bic(a,k)$.  For $i=0$ or $1$,  we set
$$
{\Om}_{\oGamma}^R(k)_i=\Bigg(\prod_{
\begin{smallmatrix} v\in \ell^{-1}(i) \\ v\in V^{\rm nab} \end{smallmatrix}} \Om^{R(v)}(b(v),k)^{\rm nab}  \Bigg)\times \Bigg( \prod_{
\begin{smallmatrix} v\in \ell^{-1}(i) \\ v\in V^{\rm ab} \end{smallmatrix}} {\Om}^{R(v)}(b(v),k)^\ab\Bigg),
$$
where for all $v\in V(\Gamma)$:
\begin{itemize}
\item $b(v)$ is the vector of twists at the half-edges adjacent to $v$;
\item $R(v)$ is the subset of $i\in R$ of indices adjacent to $v$.
\item The space $ \Om^{R}(a,k)^{\rm ab}$ is  the sub-space of $ \Om^{R}(a,k)$ of differentials obtained as powers of meromorphic $1$-forms and $\Om^{R}(a,k)^{\rm nab}$ is the complement of $\Om^{R}(a,k)^{\rm ab}$.
\end{itemize} 
At the level $-1$,  we define a space $\widetilde{\Om}_{\oGamma}^R(k)_{-1}$ as the sub-stack of ${\Om}_{\oGamma}^R(k)_{-1}$ of $k$-differentials satisfying the {\red global residue} condition of~\cite[Definition~1.4]{BCGGM2}.  We recall the description of the global residue under two different conditions that will be relevant to us.

\sskip

$(\star)$ The vector $a$ is positive, $\overline{\Gamma}$ is in $\dStar(a,k)$ and the edges between vertices of $V^{\rm nab}$ are not integral.  In this situation we have
$\widetilde{\Om}_{\oGamma}^R(k)_{-1}={\Om}^{R_{-1}}(b(v_{-1}),k),$
where $v_{-1}$ is the unique vertex of level $-1$  and $R_{-1}$ is the set of half-edges leading to a vertex of $V^{\rm ab}$.   

\sskip

$(\star\star)$  The graph $\overline{\Gamma}$ is in $2-\Bic(a,k)$ and there is a leg $i\in \{1,\ldots,n\}\setminus R$ such that $i$ is incident to the vertex of level 0.  Under this assumption $\widetilde{\Om}_{\oGamma}^R(k)_{-1}={\Om}_{\oGamma}^R(k)_{-1}$ (the global residue condition is trivial in this case).  

\sskip

In all cases we set 
\begin{eqnarray}
{\Om}_{\oGamma}^R(k) \coloneqq {\Om}_{\oGamma}^R(k)_0 \times \PP \widetilde{\Om}_{\oGamma}^R(k)_{-1}.
\end{eqnarray}
There is a canonical morphism $\zeta_{\oGamma}:\PP{\Om}_{\oGamma}^R(k)\to \PP{\oOm}^R(a,k)$ defined by mapping a point of ${\Om}_{\oGamma}^R(k)$ to  a differential on a nodal curve that vanishes on components associated with vertices of level $-1$.  This morphism is of degree 0 if the fibers have  positive dimensions,  and $\big|{\rm Aut}(\oGamma)\big|$ otherwise.   We denote  by ${\rm Irr}_\oGamma^R(k)$  the set of irreducible components of $\PP{\Om}_{\oGamma}^R(k)$ such that  $\zeta_{\oGamma}$ is finite along $D$.

\subsection{Relations in the cohomology of $\PP\oOm^R(a,k)$} \label{sec:relations} 
Let $i \in \{1,\ldots,n\}$, and let ${\rm X}$ be one the sets: $\Bic$, $2-\Bic$ or $\dStar$.  We denote by ${\rm X}_i(a,k)$  the set of graphs such that the label $i$ is incident to a vertex of level $-1$.  If $i\notin R$, and $a_i$ is a negative integer, then we denote by ${\rm X}_i^R(a,k)$ the set of graphs such that either:
\begin{itemize}
    \item the  leg $i$ is incident to a vertex of level $-1$,
    \item or the leg $i$  is incident to a vertex $v$ of level 0 in $V^{\rm ab}$, and all other legs incident to $v$ are either in $R$ or have positive twist.
\end{itemize} 
With this notation, the set ${\rm Bic}_i^R(a,k)$ is the set of bi-colored graphs such that the $i$-th residue vanishes identically along the associated strata. 

\begin{theorem}\label{th:relpsi}
Let $i \in \{1,\ldots,n\}$.  For all $\oGamma\in {\rm Bic}_i(a,k)$,  and all  $D \in {\rm Irr}(\oGamma)$ there exists an integer $m_D$ such that:
\begin{eqnarray}\label{for:relclass1}
\xi + k  a_i \,  \psi_i\, = \sum_{
\begin{smallmatrix}
\oGamma \in {\rm Bic}_i(a,k)\\
D \in {\rm Irr}_\oGamma^R(k)
\end{smallmatrix}
} \frac{m_D}{\big|{\rm Aut}(\oGamma)\big|} \zeta_{\oGamma*}[D],
\end{eqnarray}
and if $i\notin R$, and $a_i$ is a negative integer,   then:
\begin{eqnarray}\label{for:relclass2}
 \xi\, = \, k\,  [\PP\oOm^{R\cup \{i\}}(a,k)] + \sum_{
\begin{smallmatrix}
\oGamma \in {\rm Bic}^R_i(a,k)\\
D \in {\rm Irr}_\oGamma^R(k)
\end{smallmatrix}
} \frac{m_D}{\big|{\rm Aut}(\oGamma)\big|}\cdot \zeta_{\oGamma*}[D].
\end{eqnarray}
If either condition $(\star)$ or $(\star\star)$ of Section~\ref{ssec:strata} is satisfied then  $m_D=k^{|E(\Gamma)|}m(\oGamma)$.
\end{theorem}

\begin{proof} To establish relation~\eqref{for:relclass1},  we consider the line bundle $\mathcal{O}(1)\otimes \mathcal{L}_i^{m_i}\to \PP\oOm^E(a,k)$. This line bundle has a section defined by 
$$s_i:\eta\mapsto \text{$m_i$th order of $\eta$ at $x_i$}.$$
This section does not vanish on $\PP\Om^{R}(a,k)$,  nor on strata associated with decorated graphs of depth $0$ or  bi-colored graphs in ${\rm Bic}(a,k)\setminus {\rm Bic}_i(a,k)$.  Therefore, up to co-dimension $2$ loci,  the vanishing locus of $s_i$ is the union of the images of irreducible component $D$ of ${\rm Irr}_\oGamma^R(k)$ for  $\oGamma$ in ${\rm Bic}_i(a,k)$.  Thus,  there exist integers $m_D$ such that  the relation~\eqref{for:relclass1} is satisfied.

\sskip

If $a_i$ is a negative integer and $i$ is not in $R$,  then we  consider the line bundle $\cO(1)$ and its section given by the $i$-th residue morphism ${\rm res}_i$.  This section vanishes along $\PP\oOm^{R\cup \{i\}}(a,k)$ with multiplicity $k$ (Lemma~\ref{lem:vanres}).  Besides,  the function ${\rm res}_i$ does not vanish along strata associated with decorated graphs of depth $0$ or bi-colored graphs that are not in ${\rm Bic}^R_i(a,k)$.  Therefore,  there exist integers $\widetilde{m}_D$ satisfying:
\begin{eqnarray*}
 \xi\, = \, k\,  [\PP\oOm^{R\cup \{i\}}(a,k)] + \sum_{
\begin{smallmatrix}
\oGamma \in {\rm Bic}_i(a,k)\\
D \in {\rm Irr}_\oGamma^R(k)
\end{smallmatrix}
} \frac{\widetilde{m}_D}{\big|{\rm Aut}(\oGamma)\big|}\cdot \zeta_{\oGamma*}[D].
\end{eqnarray*}
We have to prove that $\widetilde{m}_D=m_D$ if $D$ is in ${\rm Bic}_i(a,k)$.  Let $\Delta$ be an open disk of $\CC$ parametrized by $\epsilon$.  Let  $\Delta\to \PP\oOm^R(a,k) \setminus \PP\oOm^{R\cup \{i\}}(a,k)$  be a morphism such that image of $\epsilon=0$ is a  point of the image of  $D$,  while the other points are mapped to $\PP\Om^R(a,k)$.    Up to a choice of a smaller disk, there exists a lift of $\epsilon$ to $\cO(-1)^*$.  Then, there exists an integer $\ell$  and holomorphic functions $f$ and $\widetilde{f}$ that do not vanish  on $\Delta$ such that $s_i=\epsilon^{\ell} f$ and ${\rm res}_{i}= \epsilon^{\ell} \widetilde{f}$  (see the ``necessary'' part of Theorem~1.5  of~\cite{BCGGM2}).  Thus $s_i$ and ${\rm res}_{i}$ vanish with the same multiplicity $\ell$ along $\epsilon=0$.  This implies that the vanishing multiplicity of $s_i$ and ${\rm res}_{i}$ along any branch of the divisor $D$ are equal and $\widetilde{m}_D=m_D$.  

\sskip

\noindent{\textit{Computation of $m_D$ for $k$-star graphs.}} We set
\begin{eqnarray*}
{\rm lcm}(\oGamma,k)&\coloneqq&{\rm lcm}(kb(e))_{e\in E(\Gamma)}, \text{ and}\\
G(\oGamma,k)&\coloneqq&\left(\prod_{e\in E(\Gamma)} \mathbb{U}_{kb(e)} \right)\bigg/  \mathbb{U}_{{\rm lcm}(\oGamma,k)}.
\end{eqnarray*}
The fact that $m_D=k^{|E(\Gamma)|}m(\oGamma)$ in the situations $(\star)$ and $(\star\star)$ is a direct consequence of the following  Lemma. 
\end{proof}
\begin{lemma}\label{lem:technical} If condition $(\star)$ or $(\star\star)$ is satisfied and $y$ is a point of $\PP\Omega_\oGamma^R(k)$,  then there exists an open neighborhood $U$ of $y$ in $\PP\Omega_\oGamma^R(k)$,  a disk $\Delta$ in $\CC$ containing $0$  and a morphism  $\iota\colon U\times \Delta\times G(\oGamma,t)\to \PP\oOm^R(a,k)$ satisfying:
\begin{itemize}
\item For all $\gamma \in G(\oGamma,k)$, the morphism $\iota$ induces an isomorphism $U\times \{0\} \times g$ with $U$. 
\item The image of $U\times (\Delta\setminus\{0\}) \times G(\oGamma,k)$ lies in $\PP\Om^R(a,k,E)$.
\item The section $s_i$ vanishes with multiplicity ${\rm lcm}(\oGamma,k)$ along $U\times \{0\} \times G(\oGamma,k)$.
\item The morphism $\iota$ is a degree $1$ parametrization of a neighborhood of $y$ in $\PP\oOm_\oGamma^R(a,k)$. 
\end{itemize}
\end{lemma}

{\red
\begin{remark} The main difficulty in the description of the boundary of moduli spaces of $k$-differentials (compared to the case $k=1$) is the fact the global residue condition depends on an extra datum: the choice of a $k$-cover of the stable graph. We did not introduced this notion here because it is irrelevant if we impose the conditions  $(\star)$ or $(\star\star)$.  We restricted the computation of $m_D$ the boundary components of this type for two reasons:
\begin{itemize}
    \item we are only interested in specific integrals over moduli spaces of $k$-differentials, and we propose an inductive algorithm that will only involves graphs in $\dStar$ and $2-\Bic$ (and avoids the general bi-colored graphs);
    \item we need to control the large $k$-asymptotic of these integrals. 
\end{itemize} 
If one wants to compute a class (and not just an integral), then one needs to compute all the coefficients $m_D$ and these were later computed in~\cite{CosMolSch}. However, we should emphasize that it is still unclear how to control the large $k$ asymptotic of such classes if we write formulas involving all bi-colored graphs. Indeed, the classes associated to components corresponding to different $k$-covers can be different, so one would need bounds on the norm of these classes which are independent of this choice.
\end{remark}}

\begin{proof}[Proof of Lemma~\ref{lem:technical}] The proof is similar to the proof of Lemma~5.6 of~\cite{Sau}. 
We can decompose the point $y$ into
$$y=y_0 \times y_{-1}=(C_0,[\eta_0]) \times (C_{-1},[\eta_{-1}]) \in  \PP\Omega_{\oGamma}^R(k)_{0}\times \PP\Omega_{\oGamma}^R(k)_{-1},$$
where $\eta_i$ is a $k$-differential up to a scalar (we omit the notation of the markings). For $i=0$ and $ -1$, we choose a neighborhood $U_i$ of $y_i$ in  $\PP\Omega_{\oGamma}^R(k)_{i}$ together with a trivialization $\sigma_i$ of $\mathcal{O}(-1)\to \PP\Omega_{\oGamma}^R(k)_{i}$. We assume that $U=U_{0}\times U_{-1}$  has coordinate $u=(u_0,u_{-1})$ and that $y=(0,0)$. We rephrase the choice of a  trivialization of the line bundle as: we choose a family of $k$-differentials $(C_i(u_i),\eta_i(u_i))$  for $u_i\in U_i$ such that $(C_i(0),[\eta_i(0)])=y_i$ for $i=0$ or $-1$.

\sskip

Let $e=(h,h')$ be an edge of $\Gamma$ with twist $kb(e)$. Let $\sigma_0\colon U\to C_{0}$ and $\sigma_{-1}\colon U\to C_{-1}$ be the sections corresponding to the branch of the node associated with $e$. For $i=0,-1$, there exists a neighborhood $V_i$ of $\sigma_i$ in $C_i$ of the form $U_i\times \Delta_{e,i}$ where $\Delta_{e,i}$ is a disk of the plane parametrized by $z_{e,i}$,  and such that 
$$\eta_{i}(u_i,z_{e,i})= \left(z_{e,i}^{\pm kb(e)}+{\rm res}_{e, i}(u_i) \right) \left(\frac{dz_{e,i}}{z_{e,i}}\right)^k,
$$
where the sign is positive for $i=0$ and negative for $i=-1$,  and the residue function ${\rm res}_{e, i}$  is zero unless $i=-1$, $b(e)$ is integral, and we are in the $(\star\star)$ situation.  The coordinates $z_{e,i}$ are only defined up to a $kb(e)$-th root of unity.  We fix such a choice for all edges $e$ and $i=0,-1$.

\sskip

\noindent{\textit{Constructing a smoothing of $\eta$ under condition $(\star)$.}} For all $e\in E(\Gamma)$, we fix $\zeta_e$ a $kb(e)$-th root of unity, i.e. an element $\zeta$ in  $\prod_{e\in E(\Gamma)} \mathbb{U}_{kb(e)}$.  With this data, we construct a family of curves  $C_\zeta\to \Delta\times U$ as follows.  Around a node corresponding to $e\in E(\Gamma)$, we define $C_{\zeta}(\epsilon,u)$ as the solution of
$$z_{e,0} \cdot z_{e,-1}=\zeta_e\cdot\epsilon^{{\rm lcm}(\oGamma,k)/(kb(e))}
$$
in $\Delta_{e,0}\times \Delta_{e,-1}$. Outside a neighborhood of the nodes, we define  $C_{\zeta}(u,\epsilon)\simeq C_{0}(u)$ or $C_{-1}(u)$. On this family of curves, we can define a $k$-differential by 
$$\eta_\zeta=z_{e,0}^{kb(e)}\left(\frac{dz_{e,0}}{z_{e,0}}\right)^k =  \frac{{\epsilon}^{{\rm lcm}(\oGamma,k)}}{z_{e,-1}^{kb(e)}}\cdot \left(\frac{dz_{e,-1}}{z_{e,-1}}\right)^k 
$$
in the chart $z_{e,0}z_{e,-1}=\zeta_e\cdot\epsilon^{{\rm lcm}(\oGamma,k)}$. Then this differential is extended by $\eta_0$ or $\epsilon^{-\rm lcm}\eta_{-1}$ outside a neighborhood of the nodes.   

\sskip

\noindent{\textit{Constructing a smoothing of $\eta$ under condition $(\star\star)$.}}  We have to modify the previous construction as  the residue function is non-trivial.  We fix a $k$-th root of the $k$-differential at the upper vertex.  We assume that the residues of $\eta_{-1}$ at the node are non-trivial (generic situation) and we fix locally a $k$-th root $r_e^{1/k}$ of each of these residue functions.  With this data, we can find a family of $1$-differentials $(\widetilde{C_0}, {v_0})$ parametrized by  $U_0\times \Delta$ such $\widetilde{C}_0(u_0,0)=C_0(u_0)$, and $v_0^k(u_0,0)=\eta_0(u_0)$, while the residue of $v_0$ at the node corresponding to an edge $e$ is given by $\epsilon^{{\rm lcm}(\oGamma,k)} r_e^{1/k}$(see~\cite[Proposition~3.10]{Sau}).  On this family of differentials, we can find a coordinate $z_{e,0}$ such that $v_0^k=(z_{e,0}^{kb(e)}+r)\left(\frac{dz_{e,0}}{z_{e,0}}\right)$ on an annulus around 0, and the smoothing construction still works. 

\smallskip

\noindent{\textit{Neighborhood of the boundary.}} Two deformations $(C_{\zeta},\eta_\zeta)$ and $(C_{\zeta'},\eta_{\zeta'})$ are isomorphic if and only if $\zeta=\rho \zeta'$ for some ${\rm lcm}(\oGamma,k)$-th root of unity $\rho$. Therefore the morphism:
\begin{eqnarray*}
\iota\colon \PP(U) \times \Delta \times G(\oGamma,k)&\to& \PP\oOm^R(a,k)\\
(u, \epsilon, \gamma) &\mapsto & (C_{\gamma}(u,\epsilon),\eta_{\gamma}(u,\epsilon))
\end{eqnarray*}
is of degree $1$ on its image.  To check that this morphism parametrizes a neighborhood of $y$, we can show as in the case of abelian differentials that there exists a retraction  $\eta_V:\widetilde{V}\to V$, where $\widetilde{V}$ is 
a neighborhood of $y$ in $\PP\oOm^R(a,k)$. Besides, all points $y'$ in $V$ lies in the image of $\{\eta(y')\}\times \Delta\times G(\oGamma,k)$  under $\iota$ (see~\cite[Lemma~5.6]{Sau}, or~\cite[Proposition~1.2]{CosMoeZac}).
\end{proof}
 
\section{Growth of integrals on moduli spaces of pluri-canonical divisors}\label{sec:DR}

{\red  The main results that will be established in the next section (Theorem~\ref{th:fr} and Lemma~\ref{lem:defvpsi}) state that the integrals 
$$
\int_{\PP\oOm(a,k)} \xi^{2g-3+n}, \text{ and } \int_{\PP\oOm(a,k)} \psi_2\xi^{2g-4+n}
$$
are computable, and that we can describe their behavior for large values of $k$. One of the main difficulty to compute these integrals is that they cannot be directly as integrals over $\oM(a,k)$ (i.e. as integrals in $\oM_{g,n}$) because $\xi$ is not a canonically a  pull-back along $\PP\oOm(a,k)\to \oM(a,k)$. Our overall strategy in the next section will be to to replace $\xi$ by $\psi_1$ (thus a class of cohomology of $\oM_{g,n})$ using Theorem~\ref{th:relpsi} established in the previous section. In the present section we provide the description and estimates for the integrals $$\int_{\oM_{g,n}(a,k)}\psi_1^{2g-3+n},  \text{ and }\int_{\oM_{g,n}(a,k)} \psi_1^{2g-4+n} \psi_2$$ 
(and generalization of these integrals involving residue conditions) which will be used in the next section. We will use that $[\oM(a,k)]$ is computable and grows like $k^{2g}$. This follows from previous works~\cite{BHPSS,PixZag,Spe,CosSauSch} on double ramification cycles. 
}

\subsection{Statement of the result of the section}

Let $(a,k)$ be a rational pair of $\Delta_{g,n}$. We denote by $\M(a,k)^{\rm nab}\subset \M(a,k)\setminus \M(a,1)$, the subspace of  pluri-canonical divisors with no residues. We denote by $\oM(a,k)^{\rm nab}$ the closure of $\M(a,k)^{\rm nab}$ in $\oM_{g,n}$, and we set
\begin{eqnarray}
    \A(a,k)\coloneqq \int_{\oM(a,k)^{\rm nab}} \frac{1}{1+ka_1\psi_1} \\
    \A^2(a,k)\coloneqq \int_{\oM(a,k)^{\rm nab}} \frac{\psi_2}{1+ka_1\psi_1}
\end{eqnarray}
We define the polynomial $\A^2\colon \Delta_{g,n}\to \RR$ by the formula
\begin{eqnarray}\label{def:A2}
    a_2\A^2(a)  &=& -\A(a) + \sum_{i=3}^n (a_2+a_i-1)\A(a_1, a_2+a_i-1,\ldots,\widehat{a_i},\ldots) \\
    \nonumber
    && + \int_{b=0}^{a_2-1} \frac{b(a_2-1-b)}{2} \A(a_1,\widehat{a_2},\ldots,b,a_2-1-b)\, db,
\end{eqnarray}
where the polynomial $\A$ is defined by~\eqref{def:A}. Actually, the formula~\eqref{def:A2} only defines $\A^2$ as a rational function, but the polynomiality of $\A^2$ will be proved in the section. The following proposition is established at the end of the Section. It is the only result used in the rest of the text. 
\begin{proposition}\label{pr:growthcanonical} The rational numbers $\A(a,k)$ and $\A^2(a,k)$ can be explicitly computed. 
If $K$ is a compact set of $\Delta_{g,n}$, then there exists a positive constant $C>0$ such that 
\begin{equation}
    \left|\A(a,k)\right|<C\, k^{4g-3+n}, \text{ and }  \left|\A^2(a,k)\right|<C\, k^{4g-4+n},
\end{equation}
for all pairs in $K$. Furthermore, there exists a positive constant $C'$ such that 
\begin{equation}\label{for:growth}
    \left|k^{-4g+3-n}\A(a,k) - \A(a)\right| < \frac{C'}{k} \text{ and }\left|k^{-4g+4-n}\A^2(a,k) - \A^2(a)\right| < \frac{C'}{k}
\end{equation}
for all pairs in $K$ such that $a$ has no integral negative coordinates.
\end{proposition}

\subsection{Double ramification cycles}{\label{ssec:DR}}

A {\em simple star graph} is a twisted star graph such that the twists at edges incident to outer vertices are positive integers. We denote by ${\rm sStar}$ the set of simple star graphs. Let $(a,k)$ be rational vector of $\Delta_{g,n}\setminus \ZZ_{>0}^n$. 
\begin{definition}
    The {\em double ramification cycle} is the class in $H^{2g}(\oM_{g,n},\QQ)$ defined as 
\begin{equation}\label{for:dr}
{\rm DR}(a,k)\, =\, \sum_{(\Gamma,b) \in {\rm sStar}(a,k)} \frac{k^{|E(\Gamma)|-|{\rm Out}(\Gamma)|} m(\Gamma,b)}{\big| {\rm Aut}(\Gamma,b)\big|} \zeta_{\Gamma*}\left( [\oM(b(v_0),k)] \times \prod_{v \in {\rm Out}(\Gamma)} [\oM(b(v),1)] \right)  
\end{equation}
where $v_0$ is the central vertex and ${\rm Out}(\Gamma)$ is the set of outer vertices.
\end{definition}

By \cite[Theorem~9]{BHPSS} the class ${\rm DR}(a,k)$ is equal to a cohomology class $P(a,k)$ called {\red \em Pixton's class}. This second class can be explicitly computed (see~\cite{DelSchZel} for the implementation with Sagemath), and depends polynomially on the parameters~\cite{PixZag,Spe}.  Moreover, the cohomology classes of spaces of canonical divisors may be explicitly computed in terms of double ramification cycles~\cite{FarPan,Sch}. We summarize these results in the following proposition. 
\begin{proposition} The class ${\rm DR}(a,k)$ is a polynomial in the variables $ka_1,\ldots,ka_n$ of degree $2g$ on its domain of definition. The coefficients of this polynomial are explicitly computable. Furthermore, the class $[\oM(a,k)]$ is explicitly computable for all choices of $a$ and $k$ (even if $a$ sits in $\ZZ_{>0}^n$).  
\end{proposition}
We extend the definition of ${\rm DR}(a,k)$ to vectors $a\in \ZZ_{>0}^n$  by polynomiality.  The following lemma gives an alternative expression of ${\rm DR}(a,k)$ that is valid for all vectors. To state it, we denote by ${\rm sStar}^1\subset {\rm sStar}$ the set of simple star graphs such that $x_1$ is incident to the central vertex. 
\begin{lemma} For all pairs $(a,k)$ of $\Delta_{g,n}$ the  following relation {\red
\begin{eqnarray} 
\nonumber {\rm DR}(a,k) &=&  [\oM(a,k)] - \delta_{a\in \ZZ_{>0}^n} (1+a_1\psi_1)  [\oM(a,1)]  \\ \label{for:dr1}
&&+ \!\!\!\!\! \!\! \!\! \underset{\text{non-trivial}}{\sum_{(\Gamma,b) \in {\rm sStar}_1(a,k)}} \!\! \!\! \frac{k^{|E(\Gamma)|-|{\rm Out}(\Gamma)|} m(\Gamma,b)}{\big| {\rm Aut}(\Gamma,b)\big|} \zeta_{\Gamma*}\left( [\oM(b(v_0),k)] \times \!\!\!\!\! \prod_{v \in {\rm Out}(\Gamma)} \!\!\! [\oM(b(v),1)] \right).
\end{eqnarray}}
\end{lemma}

\begin{proof}
If $a$ is a vector of $\Delta_{g,n}$, then we denote $a'$ the vector $(0,a_1+1,\ldots,a_n)$ of $\Delta_{g,n+1}$. This vector is not in $\ZZ_{>0}^{n+1}$ so the class ${\rm DR}(a',k)$ is defined by formula~\eqref{for:dr}. We denote by $\delta\in H^*(\oM_{g,n+1},\QQ)$ the class of the  divisor of nodal curves  with a genus 0 component carrying $x_0$ and $x_1$. We set $${\rm DR}'(a,k)\ = \ \pi_*\left(\delta \cdot {\rm DR}(a',k)\right),$$ where $\pi\colon \oM_{g,n+1}\to \oM_{g,n}$ is the forgetful morphism of the last marking.   The class ${\rm DR}'(a,k)$ is equal to the RHS of~\eqref{for:dr1}.  Indeed, each graph of ${\rm sStar}(a',k)$ contributes trivially unless the first leg is incident to the  central vertex. If ${\rm sStar}(a',k)$ contains a twisted graph $(\Gamma,b)$ such that $\zeta_{\Gamma *}[\oM_\Gamma]=\delta$, then the associated contribution is equal to $-a_1\psi_1[\oM(a,1)]$. The rest of the graphs contribute as the sum in the RHS.  

\sskip

Finally, the equality ${\rm DR}'(a,k)={\rm DR}(a,k)$  holds if $a_1$ is not an integer. By polynomiality of  the classes ${\rm DR}$ on ${\rm DR}'$ this equality is valid for all pairs $(a,k)$ of $\Delta_{g,n}$.  
\end{proof}

{\red
\begin{definition} We set 
\begin{eqnarray*}
    {\rm DR} \colon \Delta_{g,n}&\to& H^{2g}(\oM_{g,n},\RR) \\
    a &\mapsto& {\rm DR}^{\rm hom}(a_1,\ldots,a_n)
\end{eqnarray*}
where ${\rm DR}^{\rm hom}$ is the polynomial in $n$ variables defined  by: ${\rm DR}^{\rm hom}(a_1k,\ldots,a_nk)$ is the part of degree $2g$ of the polynomial ${\rm DR}(a,k)$ in $\QQ[ka_1,\ldots,ka_n]$. 
\end{definition}

The function ${\rm DR}$ is the unique function  that satisfies: for all compact $K\subset \Delta_{g,n}$, there exists a positive constant $C>0$ such that $\left\lVert {\rm DR}(a,k) - k^{2g} {\rm DR}(a) \right\rVert < C\, k^{2g-1}$  for all pairs of $K$. This is a polynomial function of degree $2g$ in any subset of $n-1$ variables of $\{a_1,\ldots,a_n\}$ (the last one being fixed by the constraint $|a|=2g-2+n$).}

\begin{lemma}\label{lem:growth} {\red If $K$ is a compact set of $\Delta_{g,n}$}, then there exists a positive constant $C$ such that 
$$\left\lVert [\oM(a,k)] - k^{2g} {\rm DR}(a) \right\rVert < C\, k^{2g-1}$$ for all pairs $(a,k)$ {\red of $K$}.
\end{lemma}

\begin{proof}
The proof is done by induction on $g$ and $n$. The base of induction is trivial as $[\oM_{0,n}(a,k)]=1$, so we assume that $g\geq 1$. First, there are  finitely many vectors of $\Delta_{g,n}^+$ with integral coordinates, so the degree $2g-2$ part of $[\oM(a,k)]$ is bounded. Therefore we may replace in the statement the class $[\oM(a,k)]$ by its degree $2g$ part. If we restrict ourselves to the cohomological degree $2g$ then the following relation holds
\begin{eqnarray*}
[\oM(a,k)]={\rm DR}(a,k) + a_1 \psi_1 [\oM(a,1)]  - \!\!\!\! \sum_{(\Gamma,b) \in {\rm sStar}^*_1(a,k)} c_{\Gamma,b},
\end{eqnarray*}
where ${\rm sStar}^*_1(a,k)$ is the set of non-trivial star graphs and $c_{\Gamma,b}$ is the summand defined by~\eqref{for:dr1}. The class $a_1 \psi_1 [\oM(a,1)]$ is bounded, so we study the sum in the RHS.  There are finitely many star graphs and finitely many possible values of twists at half-edges incident to outer vertices. Therefore, the number of terms in the sum in~\eqref{for:dr1} is bounded, and there exists a positive constant $C_1$ such that
$$
\lVert c_{\Gamma,b} \rVert < C_1\, k^{2g(v_0)+(|E(\Gamma)|-|{\rm Out}(\Gamma)|)}< C_1\, k^{2g-1}
$$
by induction hypothesis applied to the class at the central vertex of $\Gamma$. Combining these inequalities we get that $\left\lVert [\oM(a,k)]- {\rm DR}(a,k) \right\rVert<C\, k^{2g-1}$ for some constant $C>0$ and the lemma follows by definition of ${\rm DR}(a)$. 
\end{proof}
{\red
We can now show that the functions $\A$ and $\A^2$ defined by the coefficient extraction formulas~\eqref{def:A} and~\eqref{def:A2} can be defined as intersection numbers.
\begin{lemma}\label{lem:ADR}
For all $a\in \Delta_{g,n}$, we have
\begin{eqnarray}\label{for:Adr}
{\A}(a)\, &=& \, \int_{{\rm DR}(a)} \frac{1}{1+a_1\psi_1},\, \\\label{for:Adr2}
\text{ and }\,\, {\A}^2(a)\, &=&\, \int_{{\rm DR}(a)} \frac{\psi_2}{1+a_1\psi_1}.
\end{eqnarray}
\end{lemma}}
\begin{proof}
If $a_1$ and $a_2$ are  strictly rational, then for all graphs non-trivial $(\Gamma,b)\in {\rm sStar}(a,k)$, the legs $1$ and $2$ are incident to the central vertex. Then  the intersection of $\psi_1^{2g-3+n}$ or $\psi_2\psi_1^{2g-4+n}$ with the class defined by  $(\Gamma,b)$ in~\eqref{for:dr} vanishes for dimension reasons. Therefore, the only graph contributing to the intersection of ${\rm DR}(a,k)$ with $\psi_1^{2g-3+n}$ or $\psi_2\psi_1^{2g-4+n}$ is the trivial star graph, and we have
\begin{eqnarray*}
{\A}(a,k)\, = \, \int_{{\rm DR}(a,k)} \frac{1}{1+ka_1\psi_1},\, \text{ and }\,\, {\A}^2(a,k)\, =\, \int_{{\rm DR}(a,k)} \frac{\psi_2}{1+ka_1\psi_1}
\end{eqnarray*}

This observation ensures that the polynomial $\A$ and $\A^2$  are determined by integrals on double ramification cycles.  Then, by ~\cite[Theorem~1.1]{CosSauSch}, we have
$$
\int_{{\rm DR}(a,k)} \frac{1}{1+ka_1\psi_1} \, =\, k^{4g-3+n} [z^{2g}]  (-a_1)^{2g-3+n} {\rm exp}\left(a_1z\frac{\S'(z)}{\S(z)}\right) \prod_{i=2}^n \frac{\S(a_iz)}{\S(z/k)\S(z)^{2g-1+n}},
$$
thus  equality~\eqref{for:Adr} holds. To compute the integral $\A^2(a,k)$ we use the ``splitting formula''  \cite[Proposition 3.1]{CosSauSch}. This formula expresses the product $$k(a_2\psi_2-a_1\psi_1)\, {\rm DR}(a,k)$$ as a sum indexed by bi-colored graphs with two vertices. If we multiply this formula by $(-ka_1\, \psi_1)$, then most terms vanish apart from those for which the vertex carrying the second leg is of genus 0 with 3 half-edges. Therefore, the following relation holds
\begin{eqnarray*}
    a_2k\A^2(a,k)+\A(a,k) &=&  \sum_{0< b< k(a_2-1)} \frac{b(ka_2-k-b) }{2}\A((a_1,\widehat{a_2},\ldots, b/k, a_2-1-b/k),k)\\
    &&+ \sum_{i>2} k(a_2+a_i-1)\A((a_1,a_2+a_i-1,\ldots,\widehat{a_i},\ldots),k).
\end{eqnarray*}
The first sum of this expression is of the form 
$$\frac{k^{4g-4+n}}{2}\int_{0\leq b\leq a_2-1} b(a_2-1-b) \A(a_1,\widehat{a_2},\ldots, b, a_2-1-b)\, db \, + \, O\left(k^{4g-5+n}\right)$$
by Riemann approximation of integrals, thus~\eqref{for:Adr2} is valid.\end{proof}

\subsection{Large $k$ asymptotic in presence of residue conditions}\label{ssec:largekresidue}


If $R$ is a subset of $\{1,\ldots,n\}$, then we denote $\M^R(a,k)\subset \M(a,k)$ the locus of canonical divisors with no residue at $x_i$  for $i\in R$, and by $\oM^R(a,k)$ its closure in $\oM_{g,n}$. 
\begin{lemma}\label{lem:growth2} The class $[\oM^R(a,k)]$ can be explicitly computed for all pairs of $\Delta_{g,n}$. Moreover, for all compact sets  $K\subset \Delta_{g,n}$, there exists a positive constant $C$ such that $$\left\lVert [\oM^R(a,k)]  \right\lVert < C\, k^{2g}.$$
for all pairs $(a,k)$ of $K$.
\end{lemma}

\begin{proof} We prove this lemma by induction on the size of $R$.   The base  $R=\emptyset$ is provided by Lemma~\ref{lem:growth}.  Therefore, we fix a set $R\subset \{2,\ldots,n\}$, and we will prove that the conclusion of the lemma is valid for the set $R\cup \{1\}$. 

\sskip

We fix a compact set $K\subset \Delta_{g,n}$, and a negative integral value for $a_1$ (there are finitely many in $K$).  We apply theorem~\ref{th:relpsi} to obtain the following relation in $H^2(\PP\oOm^R(a,k),\QQ)$
\begin{eqnarray}\label{for:resind}
 [\PP\oOm^{R\cup \{1\}} (a,k)] = a_1\,  \psi_{1} -
 \sum_{
\begin{smallmatrix}
\oGamma \in {\rm Bic}_{1}^R(a,k)\setminus \Bic_{1}(a,k)\\
D \in {\rm Irr}_\oGamma^R(k)
\end{smallmatrix}
} \frac{m_D}{k\, \big|{\rm Aut}(\oGamma)\big| } \zeta_{\oGamma*}[D].
\end{eqnarray}
This relation is the difference between~\eqref{for:relclass1} and~\eqref{for:relclass2} (divided by $k$). We denote by $2-{\rm Bic}(a,k)$ the set $2-{\rm Bic}_{1}^R(a,k)\setminus 2-\Bic_{1}(a,k)$: it is the set of graphs in $2-{\rm Bic}$ such that $1$ is incident to the level 0, and no $i \in \{2,\ldots, n\}\setminus R$ such that $a_i$ is a negative integer is incident to the level 0. The graphs of $2-{\rm Bic}(a,k)$ satisfy the condition $(\star\star)$ of Section~\ref{ssec:strata}, so we can compute the multiplicity of their contribution to~\eqref{for:relclass1} and~\eqref{for:relclass2}. We consider the morphism $p\colon\PP\oOm^R(a,k)\to \oM_{g,n}$ defined by forgetting the differential. We will show that after pushing forward  formula~\eqref{for:resind} along $p$, we obtain:
\begin{eqnarray}\label{for:resind2}
 [\oM^{R\cup \{1\}}(a,k)] &=& a_1\,  \psi_{1}   [\oM^{R}(a,k)] -
 \sum_{\oGamma \in 2-{\rm Bic}(a,k)
} \frac{k^{|E(\Gamma)|-1}m(\oGamma)}{\big|{\rm Aut}(\oGamma)\big| } p_*\zeta_{\oGamma*}[\PP\Om_\oGamma^R(k)].
\end{eqnarray}

To prove this identity, we need to check that a graph $\oGamma$ of ${\rm Bic}_{1}^R(a,k)\setminus \Bic_{1}(a,k)$ defines a cohomology class that is mapped to $0$ under $p_*$ unless $\oGamma$ is in $2-{\rm Bic}(a,k)$. First, the morphism $p$ has fibers of positive dimension along the strata associated with graphs with more than one vertex of level 0. Thus, the class of any component of ${\rm Irr}_\oGamma^R(k)$ is sent to $0$ by $p_*$ unless $\oGamma$ has one vertex of level 0. If this is the case, then we observe that the first leg is incident to the unique vertex of level 0 by definition of ${\rm Bic}_{1}^R(a,k)$, thus the global residue condition is trivial. This implies that ${\rm Irr}_\oGamma^R(k)$ is empty unless $\oGamma$ has exactly one vertex of level $-1$ and~\eqref{for:resind2} holds.

\sskip 

The terms in the sum of the RHS of~\eqref{for:resind2} are determined by classes of strata with fewer residue conditions so the class $[\oM^{R\cup \{1\}}(a,k)]$ can be explicitly computed by induction hypothesis. To control the growth of $[\oM^{R\cup \{1\}}(a,k)]$,  we observe that there are finitely many values of the twists at half-edges incident to the vertex of level 0 for all graphs of $2-{\rm Bic}(a,k)$ and all pairs $(a,k)$ of $K$.  Therefore,  the number of graphs in $2-{\rm Bic}(a,k)$ is bounded for all $a\in K$, and there exists a positive constant $C_1$ such that the norm of the contribution of each graph $\oGamma$ of $2-{\rm Bic}(a,k)$ is bounded by  $C_1\, k^{2g(v_{-1})+|E(\oGamma)|-1}\leq C_1 \, k^{2g}$
by induction hypothesis (applied to the vertex that does not carry the first leg). Altogether, this implies that there exists a positive constant $C$ such that $\left\lVert [\oM^{R\cup\{1\}}(a,k)]  \right\lVert < C\, k^{2g}.$ 
 \end{proof}
{\red
 \begin{proof}[Proof of Proposition~\ref{pr:growthcanonical}]
     The first part of the Proposition follows directly from Lemma~\ref{lem:growth2}. Indeed, $[\oM(a,k)^{\rm nab}]=[\oM^R(a,k)]\setminus \delta_{a\in \ZZ^n} [\oM^R(a,1)]$ for some set $R$, and there are finitely many choices of $R\subset \{1,\ldots,n\}$, so we can bound the $[\oM(a,k)^{\rm nab}]/k^{2g}$ by the maximum of the constant provided by Lemma~\ref{lem:growth2} over all sets $R$. The second part of the proposition is given by Lemma~\ref{lem:ADR}.
 \end{proof}
}

\section{Flat recursion}\label{sec:FR}

{\red In the previous section we described certain tautological integrals involving powers of $\psi_1$ over $\oM(a,k)$, as well as their asymptotic behavior for large values of $k$ (Proposition~\ref{pr:growthcanonical}). We combine this result with Theorem~\ref{th:relpsi}, that allows for the comparison of $\xi$ and $\psi_1$ on $\PP\oOm(a,k)$, to compute integrals involving powers of $\xi$, and establish  Theorem~\ref{th:fr} and Lemma~\ref{lem:defvpsi}.}

\subsection{Growth of sums on $k$-star graphs}

Let $a$ be a vector in $\Delta_{g,n}^+$ and $\Gamma$ a star graph. {\red We recall from  Section~\ref{ssec:FR}) that $\Delta_\Gamma(a)$ is the set of twists on $\Gamma$ which are compatible with $a$, and make it a twisted star graph. We denote by $\Delta_{\Gamma}(a,k)\subset \Delta_\Gamma(a)$ the sub-set of $k$-twists. Finally we denote by $T_\Gamma(a,k)$ the quotient of $\Delta_{\Gamma}(a,k)$ by ${\rm Aut}(\Gamma)$. It is the set of non-isomorphic twists.}

\begin{lemma}\label{lem:growthgraph} We assume that $a$ is rational.  {\red Let $f:\Delta_\Gamma(a)\to \RR$ be a continuous function, and let $f_k:\Delta_\Gamma(a,k)\to \RR$ be a series of  functions.  We denote by  $\widetilde{f}_k$ the composition $\Delta_{\Gamma}(a,k)\to T_\Gamma(a,k)\overset{f}{\to} \RR$. We assume that there exists  a positive constant $C$ such that 
$$|\widetilde{f}_k(b)-f(b)|<C/k$$ 
for all $k$ and $b$ in $\Delta_\Gamma(a,k)$. Then,} we have
\begin{eqnarray*}
 \frac{1}{\big|{\rm Aut}(\Gamma)\big|} \int_{\Delta_{\Gamma}(a)}  {\red f}(b) \, db = \underset{ka \in \ZZ^n}{\lim_{k\to \infty}} \frac{1}{k^{h_1(\Gamma)}} \sum_{b\in T_\Gamma(a,k)} \frac{f_k(b)}{\big|{\rm Aut}(\Gamma,b)\big|}.
\end{eqnarray*}
\end{lemma}

\begin{proof}
As $T_\Gamma(a,k)$ is the quotient of $\Delta_\Gamma(a,k)$ by ${\rm Aut}(\Gamma)$, we can rewrite
$$
\sum_{b\in T_\Gamma(a,k)} \frac{f_k(b)}{\big|{\rm Aut}(\Gamma,b)\big|}=\sum_{b\in \Delta_\Gamma(a,k)} \frac{\widetilde{f_k}(b)}{\big|{\rm Aut}(\Gamma)\big|}.
$$ Then, the lemma follows from the convergence of Riemann sums:
$$
\underset{ka\in \ZZ^n}{\lim_{k\to \infty}} \frac{1}{{k}^{h_1(\Gamma)}} \sum_{b\in\Delta_\Gamma(a,k)} \widetilde{f}(b)  = \int_{\Delta_\Gamma(a)}  \widetilde{f}(b).
$$
\end{proof}

\subsection{Recursion relations} 
 We begin by writing a recursion relation for the intersection numbers $\V(a,k)$ for a fixed value of $k>1$.  

\begin{lemma}\label{lem:indint1} We assume that $a$ is positive. Let $\oGamma$ be a bi-colored graph in ${\rm Bic}_1(a,k)$  and let $D$ be an irreducible component of ${\rm Irr}_\oGamma(k)$ such that 
$\int_{{D}} \psi_1^{j}\xi^{2g-4+n-j} \neq 0$ 
for some non-negative integer $j$, then $\oGamma\in {\rm dStar}_1(a,k)$ and the twists at edges between vertices of $V^{\rm nab}$ are not integral (situation $(\star)$ of Section~\ref{ssec:strata}). 
\end{lemma}

\begin{proof} We assume that $\oGamma$ and $D$ satisfy the hypothesis of the lemma.  We decompose $D$ as $D_0\times D_{-1},$ where $D_0$ is an irreducible component of $\PP\Om_\oGamma(k)_{0}$, and $D_{-1}$ is an irreducible component of $\PP\widetilde{\Om}_\oGamma(k)_{0}$. We have
\begin{eqnarray*} \xi^{2g-4+n-j}\, \psi_1^{j} \, \zeta_{\oGamma *} [D]  &=&  \zeta_{\oGamma *}\left((\xi^{2g-4+n-j}[D_0]) \otimes  (\psi_1^{j}[D_{-1}])\right).
\end{eqnarray*}
Therefore ${\rm dim}(D_{-1})=j$, and ${\rm dim}(D_0)=2g-4+n-j$.
To analyse further this expression, we  identify $D_0$ with {\red $$\PP \bigg( \prod_{v\in \ell^{-1}(0)} \cO(-1)|_{D_v} \bigg),$$} where the spaces  $D_v$ are irreducible component of $\PP\oOm(b(v),k)^{\rm ab}$ or $\PP\oOm(b(v),k)^{\rm nab}$. Then, with this identification, we have
\begin{eqnarray} 
\xi^{2g-4+n-j}\cdot [D_0] &=&
\Big(\!\!\!\! \prod_{
\begin{smallmatrix} v\in \ell^{-1}(i) \\ v\notin V^{\rm ab} \end{smallmatrix}} \xi^{2g(v)-3+n(v)} [D(v)]\Big)\times \Big(\!\!\!\! \prod_{
\begin{smallmatrix} v\in \ell^{-1}(i) \\ v\in V^{\rm ab} \end{smallmatrix}} \xi^{2g(v)-2+n(v)}[D(v)]\Big).
\end{eqnarray}
We have $\xi^{2g}=0$ on $\PP\oOm(a,1)$ by~\cite[Proposition~3.3]{Sau4}, so $\xi^{2g-4+n-j}\psi_1^{j}$ vanishes along $D$ unless each vertex in $V^{\rm ab}$ has exactly one edge. This condition implies that there is only one vertex of level $-1$. Indeed the only residue conditions come from the vertices of $V^{\rm ab}$, and here these relations cannot be used to compare the residues at two different poles of level $-1$, so $D$ is of co-dimension greater than $1$ in $\PP\oOm(a,k)$ if $D$ has more than $1$ vertex of level $-1$.

\sskip

If $k\geq 2$, then $\xi^{2g-3+n}$ vanishes on $\PP\oOm(a,k)$ if $a$ has at least one integral coordinate. So the edges from vertices in $V^{\rm nab}$  to the vertex of level $-1$ have non-integral twists at the edges. Therefore the stable graph $\Gamma$ is a decorated star graph that satisfies the condition $(\star)$.
\end{proof}

If $\oGamma$ is a bi-colored graph of ${\rm dStar}_1(a,k)$ then we set 
\begin{eqnarray*}
 \V_{\oGamma}(k) &\coloneqq&  \int_{\PP\oOm_\oGamma(k)}  \frac{1}{(1-\xi)}\frac{1}{(1+ka_1\psi_1)}. 
\end{eqnarray*}
With this notation we have
\begin{equation}
\label{for:reckint}
 \V_{\oGamma}(k) \, = \, \mathcal{A}(b(v_0),k) \times  \prod_{\begin{smallmatrix}v | \ell(v)=0,\\ v\notin V^{\rm ab} \end{smallmatrix}} \!\!\!\! \V(b(v),k)  \times 
 \prod_{\begin{smallmatrix}v | \ell(v)=0,\\ v\in V^{\rm ab} \end{smallmatrix}} k^{2g(v)-2+n(v)}\V(b(v),1),
 \end{equation}
This integral vanishes unless condition~$(\star)$ is satisfied.  
\begin{lemma}\label{lem:indint2}
For all pairs $(a,k)$ of $\Delta_{g,n}^+$, we have 
\begin{eqnarray}
\label{for:V}
\V(a,k)&=& \A(a,k) + \!\!\!\!\! \sum_{\oGamma\in {\rm dStar}_1(a,k)}  \frac{k^{|E(\Gamma)|}m(\oGamma)}{\big|{\rm Aut}(\oGamma)\big|} \V_\oGamma(k). 
\end{eqnarray}
\end{lemma}

\begin{proof}
We express $\xi^{2g-3+n}$ as
\begin{equation*}
      \sum_{j\geq 0} \xi^{2g-4+n-j} (-ka_1\psi_{1})^j (\xi+ka_{1}\psi_1).
\end{equation*}
Then, we use Theorem~\ref{th:relpsi} to write 
$$
\xi+ka_1\psi_1\, =\, \sum_{\oGamma \in {\rm dStar}^*_1(a,k)}  \frac{k^{|E(\Gamma)|}m(\oGamma)}{\big|{\rm Aut}(\oGamma)\big|} \zeta_{\oGamma*}[\PP\oOm_{\oGamma}(k)] + \delta,
$$ 
where ${\rm dStar}^*_1(a,k)\subset {\rm dStar}_1(a,k)$ is the set of graphs satisfying the  condition $(\star)$, and  $\delta$ is a class supported on the strata associated to other bi-colored graphs.  By Lemma~\ref{lem:indint1} the integral of $\xi^{2g-4+n-j}\psi_1^{j}$  along $\delta$ vanishes. Besides, the integrals of these classes along graphs in ${\rm dStar}^*_1(a,k)$ are determined by formula~\eqref{for:reckint}. Therefore, we have
\begin{eqnarray*}
\V(a,k)&=& \int_{\oM(a,k)} \frac{1}{1+ka_1\, \psi_1} + \!\!\!\!\! \sum_{\oGamma\in {\rm dStar}^*_1(a,k)}  \frac{k^{|E(\Gamma)|}m(\oGamma)}{\big|{\rm Aut}(\oGamma)\big|}\V_\oGamma(k), 
\end{eqnarray*}
Finally, we can replace the summation on ${\rm dStar}^*_1(a,k)$ by a summation on the bigger set ${\rm dStar}_1(a,k)$ because $\V_{\oGamma}$ 
vanish for the graphs that do not satisfy the $(\star)$ condition.
\end{proof}

{\red With this last lemma we are now able to complete the proof of Theorem~\ref{th:fr}.}

\begin{proof}[\red End of proof of Theorem~\ref{th:fr}]
We work by induction on $g$ and $n$. The base of induction is given by $\V(a,k)=1$ if $(g,n)=(0,3)$, and $\V(a,k)=0$ if $n=1$. If we assume that $2g-2+n>1$, then Lemma~\ref{lem:indint2} expresses  $\V(a,k)$ in terms of the rational numbers $\V(a',k)$ for $|a'|<|a|$, and the rational numbers $\A(a',k)$. The first ones are computable by the induction hypothesis, while the latter ones are computable by Proposition~\ref{pr:growthcanonical}. This completes the proof of the first part of Theorem~\ref{th:fr}, and it remains to analyze the large $k$ behavior of $\V(a,k)$. 
 Let $C_1$ be a positive constant such that:
\begin{itemize}
    \item $\V(a',k)<  C_1\, k^{4g'-3+n'}$ for all $a'\in \Delta_{g',n'}^+$ such that $2g'-2+n'<2g-2+n$,
    \item and $\A(a',k)<C_1 \, k^{4g'-3+n'}$ for all vectors $a'\in \Delta_{g',n'}$ with $2g'-2+n'\leq 2g-2+n$, and with coordinates bigger than $-(2g-2+n)$.
\end{itemize} 
The existence of this constant is granted by Proposition~\ref{pr:growthcanonical} and the induction hypothesis. Then for all graphs $\oGamma$ of ${\rm dStar}_1(a,k)$, we have the inequality
\begin{equation*}
\V_\oGamma(k)\, <\, C^{|V(\Gamma)|} \, k^{d_\oGamma}, \, \text{ where }\,   d_\oGamma\, =\,\sum_{v\in V^{\rm nab}}  4g(v)-3+n(v) +\sum_{v\in V^{\rm ab}} 2g(v)-2+n(v)
\end{equation*}
obtained by combining the inequalities at each vertex of $\Gamma$. If $\Gamma$ is a star graph, then there are at most $((2g-2+n)k)^{h^1(\Gamma)}$ different twist functions making $\Gamma$ a twisted star graph, and finitely many choices of partitions $V(\Gamma)=V^{\rm nab}\sqcup V^{\rm ab}$. Besides, for all these choices, the constant $d_{\oGamma}$ is at most $d_\Gamma\coloneqq\sum_{v\in V(\Gamma)} 4g(v)-3+n(v)$. Therefore, there exists a positive constant $C_\Gamma$ such that
$$
\sum_{\oGamma}  \frac{m(\oGamma)k^{|E(\Gamma)|}}{\big|{\rm Aut}(\oGamma)\big|} \V_\oGamma(k) < C_\Gamma \, k^{d_\Gamma+|E(\Gamma)|+h^1(\Gamma)}<C_\Gamma\, k^{4g-3+n},
$$
for all pairs $(a,k)$ in $\Delta_{g,n}^+$, and where the sum is over all decorated graphs of ${\rm dStar}(a,k)$ with underlying stable graph $\Gamma$. As ${\rm Star}_{g,n}$ is finite, we conclude that there exists a positive constant $C$ such that $|\V(a,k)|<C \, k^{4g-3+n}$.

\sskip

The last step of the proof is the computation of the leading term of $\V(a,k)$. If $\oGamma$ has at least one vertex in $V^{\rm ab}$ then $d_{\oGamma}<d_{\Gamma}$, so
$$
\V(a,k) \ =\ \underset{\text{s. t. $V^{\rm ab}=\emptyset$}}{\sum_{\oGamma \in \dStar_1(a,k)}}  \frac{m(\oGamma)k^{|E(\Gamma)|}}{\big|{\rm Aut}(\oGamma)\big|} \V_\oGamma(k) + O\left(k^{4g-4+n}\right).
$$
A graph in $\dStar_1(a,k)$ with  $V^{\rm ab}=\emptyset$ is uniquely determined by its underlying twisted graph. Therefore we can write
$$
\V(a,k) \ =\ \sum_{\Gamma \in {\rm Star}_{g,n}^1} \sum_{b \in T_\Gamma(a,k)}   \frac{m(\Gamma,b)k^{|E(\Gamma)|}}{\big|{\rm Aut}(\Gamma,b)\big|} \V_{\Gamma,b}(k) + O\left(k^{4g-4+n}\right)
$$
where $\V_{\Gamma,b}(k)=\V_\oGamma(k)$ for the decorated graph determined by $(\Gamma,b)$ (and the correction term can be chosen uniformly on $\Delta_{g,n}^+$). Therefore, the theorem holds by Lemma~\ref{lem:growthgraph} and the induction hypothesis.
\end{proof}

\subsection{Two computations of $\V^2$} 
Here, we impose that $a_2$ is integral. We express the function $\V^2(a,k)$ in two different ways, which will result in two different expressions of their leading term in the large $k$ asymptotics: {\red the first one was provided by Lemma~\ref{lem:defvpsi}, while the second one (Lemma~\ref{lem:V2alterantive} below) is similar to the flat recursion and will be used in the next Section.}  

\begin{proof}[{\red Proof of Lemma~\ref{lem:defvpsi}}]
To produce the first expression of $\V^2$, we use Lemma~\ref{lem:indint1} to write
\begin{eqnarray*}
 \xi+ka_2\psi_2= \sum_{\oGamma \in {\rm dStar}^*_2(a,k)}  \frac{k^{|E(\Gamma)|}m(\oGamma)}{\big|{\rm Aut}(\oGamma)\big|} \zeta_{\oGamma*}[\PP\oOm_{\oGamma}(k)] + \delta,
\end{eqnarray*}
where $\xi^{2g-4+n}\delta=0$. Besides, $\xi^{2g-3+n}=0$ because $a_2$ is integral, so we get
\begin{eqnarray*}
  a_2\V^2(a,k)= \sum_{\oGamma \in {\rm dStar}_2(a,k)}  \frac{k^{|E(\Gamma)|-1}m(\oGamma)}{\big|{\rm Aut}(\oGamma)\big|} \xi^{2g-4+n}\, \zeta_{\oGamma*}[\PP\oOm_{\oGamma}(k)].
\end{eqnarray*}
A decorated graph contributes trivially to this sum unless it sits in $\widetilde{{\rm dStar}}_2(a,k)$, the set of decorated star graphs for which the dimension of the space defining the level $-1$ is 0. 
Therefore we may re-write  
\begin{eqnarray*}
 a_2\V^2(a,k)= \sum_{\oGamma \in \widetilde{{\rm dStar}}_2(a,k)}  \frac{k^{|E(\Gamma)|-1}m(\oGamma)}{\big|{\rm Aut}(\oGamma)\big|} \V_\oGamma(k)
\end{eqnarray*}
(here $\V_{\oGamma}$ is defined by~\eqref{for:reckint} if we exchange the roles of $1$ and $2$). 
By the same analysis as in the previous section, the joint contribution of the graphs with {\red $V^{\rm ab}\neq \emptyset$} is equal to a $O(k^{4g-5+n})$. Therefore, we can write
\begin{eqnarray*}
a_2\V^2(a,k) &=& \underset{\text{s.t. $(g(v_0),n(v_0))=(0,3)$}}{\sum_{\Gamma \in {\rm Star}_{g,n}^2}} \sum_{b\in T_\Gamma({\red a},k)} \V_{\Gamma,b} + O\left(k^{4g-5+n}\right)\\
&=& \sum_{i\neq 2} (a_2+a_i-1)\, \V((a_1,a_2+a_i-1,\ldots,\widehat{a_i},\ldots),k)\\
&& + \sum_{b_1+b_2=ka_2-k} \frac{b_1b_2}{2} \, \V(a_1,\widehat{a_2},\ldots,b_1/k,b2/k), k)\\
&& + \underset{I_1\sqcup I_2=\{1,3,4,\ldots,n\}}{ \sum_{g_1+g_2=g}} \frac{b_1b_2}{2}\, \V\left(b_1,\{a_i\}_{i\in I_1})\, \V(b_2,\{a_i\}_{i\in I_2}\right) +  O\left(k^{4g-5+n}\right),
\end{eqnarray*}
where $b_j=2g_j-2-\sum_{i\in I_i}$ in the last sum. Together with Lemma~\ref{lem:growth}, this  last identity completes the proof. 
\end{proof}

\begin{lemma}\label{lem:V2alterantive}
Let $\Gamma$ be a star graph in ${\rm Star}_{g,n}^1$. We  set $V^2_\Gamma(a)=\int_{b\in \Delta_{\Gamma}(a)}\widetilde{\V}_\Gamma^2(b)\, db$, where
\begin{equation}\label{for:V2gamma}
    \widetilde{\V}_\Gamma^2(b)= \left\{\begin{array}{cl} m(\Gamma,b)\A^2(b(v_0))\, \prod_{v \neq v_0} \V(b(v)) & \text{if $2$ is incident to the central vertex $v_0$,}\\
    m(\Gamma,b) \A(b(v_0))\,  \V^2(b(v)) \,\prod_{v'\neq v_0,v} \V(b(v'))  & \text{if 2 is incident to an outer vertex $v$.}
    \end{array} \right.
\end{equation}
Then the following expression holds
\begin{equation}\label{for:V2}
    \V^2(a)=\sum_{\Gamma \in {\rm Star}_{g,n}^1} \frac{\V^2_\Gamma}{\big|{\rm Aut}(\Gamma)\big|}.
\end{equation}
\end{lemma}

\begin{proof} The proof of this lemma is similar to the proof of Theorem~\ref{th:fr}. We write
\begin{equation*}
    \psi_2\xi^{2g-4+n}= \psi_2(\xi+ka_1\, \psi_1) \sum_{j>0} \xi^{2g-5+n-j}(-ka_1\, \psi_1)^j.
\end{equation*}
Then, we use Theorem~\ref{th:relpsi} to write 
\begin{eqnarray*}
    ka_2\, \V^2(a,k)&=& \sum_{\oGamma \in \dStar_1(a,k)}  \frac{k^{|E(\Gamma)|}m(\oGamma)}{\big|{\rm Aut}(\oGamma)\big|} \V^2_\oGamma(k), \text{ where }\\
    \V^2_\oGamma(k)&=& \int_{\PP\oOm_\oGamma(k)}  \frac{1}{(1-\xi)}\frac{\psi_2}{(1+ka_1\psi_1)}
\end{eqnarray*}
Here we have used the fact that the classes associated with graphs that do not satisfy $(\star)$ intersect trivially with the classes $\psi_2\xi^{2g-5+n-j}\psi_1^{j}$. This is proved via the same argument as in Lemma~\ref{lem:indint1}. 

\sskip

Then the integral $\V^2_\oGamma(k)$ is expressed via  a formula similar to~\eqref{for:reckint}  where we replace either $\A$ by $\A^2$, or $\V$ by $\V^2$ at the vertex that carries the second leg. The Lemma follows by the same analysis for large values of $k$ as in the proof of Theorem~\ref{th:fr}: the joint contribution of the graphs such that $V^{\rm ab}\neq \emptyset$ is equal to  a $O\left(k^{4g-5+n}\right)$, while the rest of this expression is given by 
$$
 \V^2(a,k)=k^{4g-4+n}\sum_{\Gamma \in {\rm Star}_{g,n}^1} \frac{\V^2_\Gamma(a)}{\big|{\rm Aut}(\Gamma)\big|} + O\left(k^{4g-5+n}\right)
$$
by Lemma~\ref{lem:growth}. 
\end{proof}

\section{Topological recursion at integral angles}\label{sec:TR}

{\red Here, we assume that $n\geq 2$, and we identify the set $\Delta_{g,n}$ with $\RR^{n-1}$ with coordinates $a_2,\ldots,a_n$, while  $a_1$ is the function $(a_i)_{i\geq 2} \mapsto 2g-2+n-\sum_{i=2}^n a_i$. The purpose of the section is to complete the proof of Theorem~\ref{th:psider},  i.e. the identity
\begin{equation}\label{for:psider2}
\left(\frac{\partial \V}{\partial a_2}-\V^2\right)\bigg|_{a_2=N} = 0.
\end{equation}
To establish this identity, we begin by using the expression of $\V$ and $\V^2$ as sums over star graphs~\eqref{for:FR} and~\eqref{for:V2}, that were proved in the previous section, to re-write~\eqref{for:psider2} as    
\begin{equation}\label{for:psider2b}
\sum_{\Gamma\in {\Star}_{g,n}^{1}} \frac{1}{|\rm Aut(\Gamma)|}\left(\frac{\partial \V_\Gamma}{\partial a_2}-\V_\Gamma^2\right)\bigg|_{a_2=N} = 0.
\end{equation}
In this expression, the functions $\V_{\Gamma}$ and $\V_{\Gamma}^2$ are expressed in terms of functions $\V$ and $\V^2$ in lower values of $(g,n)$, and polynomial functions $\A$ and $\A^2$ with closed expression in terms of hyperbolic functions. This reformulation, allows for a proof by induction on $(g,n)$ that is presented in the first sub-section. This proof is based on an hyperbolic identity that will be proved in the last sub-section.   
}

\subsection{Comparison of expressions on star graphs} We prove the identity~\eqref{for:psider2} by induction on $g$ and $n$. The base case  is trivial because $\Delta_{0,3}^+$ contains no vectors with integral coordinates, so we assume that $2g-2+n>1$. As explained in the introductory paragraph we use expressions of $\V$ and $\V^2$ as sums on star graphs. {\red We begin by splitting the sum in the left-hand side of~\eqref{for:psider2b} as 
$$
\sum_{\Gamma\in {\Star}_{g,n}^{1,2}} \frac{1}{|\rm Aut(\Gamma)|}\left(\frac{\partial \V_\Gamma}{\partial a_2}-\V_\Gamma^2\right)\bigg|_{a_2=N} + \sum_{\Gamma\in {\Star}_{g,n}^{1}\setminus {\Star}_{g,n}^{1,2}} \frac{1}{|\rm Aut(\Gamma)|}\left(\frac{\partial \V_\Gamma}{\partial a_2}-\V_\Gamma^2\right)\bigg|_{a_2=N}, 
$$
where ${\Star}_{g,n}^{1,2}$ is the set of star graphs such that the central vertex carries the legs $1$ and $2$. We will prove that each of these sums is 0 separately. 
}

\sskip 

First, let $\Gamma$ be a star graph of {\red ${\rm Star}_{g,n}^{1}\setminus {\Star}_{g,n}^{1,2}$. As the second leg lies on an outer vertex $v$, we write} 
$$
\V_\Gamma = \int_{b\in \Delta_\Gamma(a)} \V(b(v)) Q(a,b)\, db, \,\, \text{ and }\,\,  \V_\Gamma^2 = \int_{b\in \Delta_\Gamma(a)} \V^2(b(v)) Q(a,b)\, db.
$$
The polynomial $Q$ is the product of the factors associated with the vertices different from $v$, and the product of the twists in~\eqref{for:FRgamma} and~\eqref{for:V2gamma}. With this notation we have
\begin{eqnarray*}
    \frac{\partial \V_\Gamma(b(v))}{\partial a_2} \bigg|_{a_2=N} 
    &=& \int_{b\in \Delta_\Gamma(a)} \left(\frac{\partial \V(b(v))}{\partial a_2} Q(a,b) + \V(b(v)) \frac{\partial Q(a,b)}{\partial a_2}\right) \bigg|_{a_2=N}   \, db\\
    &=& \int_{b\in \Delta_\Gamma(a)} \left(\frac{\partial \V(b(v))}{\partial a_2} Q(a,b) \right)\bigg|_{a_2=N}   \,db\\
    &=& \int_{b\in \Delta_\Gamma(a)} \Big(\V^2(b(v)) Q(a,b) \Big)\bigg|_{a_2=N}   \, db =\V_\Gamma^2(a)\bigg|_{a_2=N}.
\end{eqnarray*}
 We used the vanishing of $\V(b(v))$ when $a_2=N$ (from the first line to the second) and the induction hypothesis (from the second line to the third). Thus, {\red we have 
$$ \sum_{\Gamma\in {\Star}_{g,n}^{1}\setminus {\Star}_{g,n}^{1,2}} \frac{1}{|\rm Aut(\Gamma)|}\left(\frac{\partial \V_\Gamma}{\partial a_2}-\V_\Gamma^2\right)\bigg|_{a_2=N}=0.$$

 Therefore, it remains to prove that the sum over ${\Star}_{g,n}^{1,2}$ also vanishes. This will be the most technical step because the terms of this sum do not vanish in general. Instead, we will use the following lemma on the polynomials $\A$ and $\A^2$ defined by hyperbolic functions, and whose proof is the object of the next sub-section.}  

\begin{lemma}\label{lem:Ahyperbolic} If we denote $c_j=[z^{2j}] \frac{z/2}{{\rm tanh}(z/2)}$ for $j\geq 1$, then
\begin{eqnarray*}
    \left( \frac{\partial \A}{\partial a_2}-\A^2\right)\bigg|_{a_2=N}&=& \sum_{0<j<N/2} c_{j,N} \, \A(a_1,N-2j,\ldots), \text{ where}\\
    c_{j,N}&=&c_j \left(\prod_{i=1}^{2j} N-i\right).
\end{eqnarray*}
\end{lemma}

\begin{proof}[{\red End of the proof of Theorem~\ref{th:psider} assuming Lemma~\ref{lem:Ahyperbolic}}]
Let $\Gamma$ be a star graph of ${\Star}_{g,n}^{1,2}$. then we write
$$
\V_\Gamma = \int_{b\in \Delta_\Gamma(a)} \A(b(v)) P(a,b)\, db, \,\, \text{ and }\,\,  \V_\Gamma^2 = \int_{b\in \Delta_\Gamma(a)} \A^2(b(v)) P(a,b)\, db,
$$
where $P$ is defined as the product of the factors associated with outer vertices and the product of the twists defined in~\eqref{for:FRgamma} and~\eqref{for:V2gamma}.  Then, after applying Lemma~\ref{lem:Ahyperbolic},  the following identity holds
$$
\left( \frac{\partial \V_\Gamma}{\partial a_2}-\V_\Gamma^2\right)\bigg|_{a_2=N} = \sum_{0<j<N/2} c_{j,N} \V_{\Gamma[j]}(a_1,N-2j,\ldots),
$$
{\red where $\Gamma[j]$ is the graph obtained from $\Gamma$ by diminishing the genus of the central vertex by $j$. Note that the automorphism group of $\Gamma[j]$ is the same as the automorphism group of $\Gamma$. Therefore we can use this identity to write 
\begin{eqnarray*}
\sum_{\Gamma \in {\Star}_{g,n}^{1,2}}  \frac{1}{\big| {\rm Aut}(\Gamma)\big|}\left( \frac{\partial \V_\Gamma}{\partial a_2}-\V_\Gamma^2\right)\bigg|_{a_2=N} &=& \sum_{\Gamma \in {\Star}_{g,n}^{1,2}} \sum_{0<j<N/2}    \frac{c_{j,N}}{\big| {\rm Aut}(\Gamma[j])\big|} \V_{\Gamma[j]}(a_1,N-2j,\ldots)\\
&=& \sum_{0<j<N/2} \sum_{\Gamma \in {\Star}_{g-j,n}^{1,2}}     \frac{c_{j,N}}{\big| {\rm Aut}(\Gamma)\big|} \V_{\Gamma}(a_1,N-2j,\ldots).
\end{eqnarray*}
 Beware that in the first line, $\Gamma[j]$ can ill-defined as it can have a vertex of negative genus. However, the associated contribution is trivial because the integration domain is empty around $a_2= N$. We use this equality to write}
\begin{eqnarray*}
\sum_{\Gamma \in {\Star}_{g,n}^{1,2}}  \frac{1}{\big| {\rm Aut}(\Gamma)\big|}\left( \frac{\partial \V_\Gamma}{\partial a_2}-\V_\Gamma^2\right)\bigg|_{a_2=N} &=& \sum_{0<j<N/2}  \sum_{\Gamma \in {\Star}_{g-{\red j},n}^{1,2}}  \frac{c_{j,N}}{\big| {\rm Aut}(\Gamma)\big|} \V_\Gamma(a_1,N-2j,\ldots)\\
&=& \sum_{0<j<N/2} \sum_{\Gamma \in {\Star}_{g-{\red j},n}^{1}}  \frac{c_{j,N}}{\big| {\rm Aut}(\Gamma)\big|} \V_\Gamma(a_1,N-2j,\ldots)\\
&=& \sum_{0<j<N/2} c_{j,N} \V(a_1,N-2j,\ldots)=0.
\end{eqnarray*} 
{\red Here we have used repeatedly the vanishing of $\V$ at vectors with an integral coordinate: from the first to the second line we have added terms with trivial contribution because the second leg lies on an outer vertex, and from the second line to the third we have applied the flat recursion formula at vectors $(a_1,N-2j,\ldots)$.} \end{proof}

\subsection{Hyperbolic identities} The proof of Lemma~\ref{lem:Ahyperbolic} relies on a series of identities satisfied by the formal series $\S$ and $\C=\frac{z\S'}{\S}+1=\frac{{\rm tanh}(z/2)}{z/2}$ in $\CC[[z]]$.
If $x$ is a variable in $\CC$,  then we use the simplified notation $\S(x)=\S(xz)$, and $\C(x)=\C(xz).$ 
For all $(a_1,\ldots,a_n)\in \CC^n$, the following formulas hold 
\begin{eqnarray}\label{for:hyperbolicproduct}
    \frac{|a|\frac{z}{2} \cS(|a|)}{\S(a_1)\ldots \S(a_n)}\!&=&\!\sum_{j\geq 0} \underset{|I|=2j+1}{\sum_{I\subset \{1,\ldots,n\}, \text{ s.t.}}} \left(\prod_{i\in I} a_i\frac{z}{2}\right)\times \left(\prod_{i\notin I} \C(a_i)\right), \text{ and}\\
    \label{for:hyperbolicproduct2}
    \frac{\C(|a|)\S(|a|)}{\S(a_1)\ldots \S(a_n)}\! &=& \!\sum_{j\geq 0} \underset{|I|=2j}{\sum_{I\subset \{1,\ldots,n\}, \text{ s.t.}}} \left(\prod_{i\in I} a_i\frac{z}{2}\right)\times \left(\prod_{i\notin I} \C(a_i)\right).
\end{eqnarray}
(these formulas are derived from product formulas for ${\rm sinh}$ and ${\rm cosh}$). 

\begin{lemma}\label{lem:hyperbolic}
For all positive integers $N$, we have
\begin{equation}\label{for:hyperbolicproduct3}
N(N-1)\left(\S(N)-\S(N-1)\C\S\right)= N\sum_{j\geq 1} c_{j,N} \,  z^{2j}\S^{2j}\S(N-2j).
\end{equation}
\end{lemma}

\begin{proof} {\red The identity~\eqref{for:hyperbolicproduct} allows for the expression of the ratios $\frac{\S(N)}{\S^N}$ and $\frac{\S(N-1)}{\S^{N-1}}$ as sums involving monomials in $\C$ and $z$. Therefore, we rewrite the LHS of~\eqref{for:hyperbolicproduct3}}  as
\begin{eqnarray*}
 &&\!\!\!\!\!\! (N-1) \S^{N}\sum_{j\geq 0} \binom{N}{2j+1}(z/2)^{2j}\C^{2N-(2j+1)} - N \S^{N}\C \sum_{j\geq 0} \binom{N-1}{2j+1}(z/2)^{2j}\C^{2N-1-(2j+1)}\\
&& =\,\S^N\sum_{j\geq 1}  ((N-1)-(N-1-2j))  \binom{N}{2j+1} (z/2)^{2j} \C^{N-(2j+1)}\\
&&= \, \S^N\sum_{j\geq 1}  2j\,  \binom{N}{2j+1} (z/2)^{2j}\C^{N-(2j+1)}
\end{eqnarray*}
{\red Similarly, the identity~\eqref{for:hyperbolicproduct} allows for the RHS of~\eqref{for:hyperbolicproduct3} to be re-written as}
\begin{eqnarray*}
    &&\S^N\sum_{j\geq 1} c_jz^{2j} \sum_{k\geq 0} \left(\prod_{i=0}^{2j-1} N-i \right) \binom{N-2j}{2k+1} (z/2)^{2k}\C^{N-2j-2k-1}\\
    &&=  \S^N \sum_{j\geq 1}\sum_{k\geq 0} c_jz^{2j} \binom{N}{2j+2k+1} \frac{(2j+2k+1)!}{(2k+1)!}(z/2)^{2k} \C^{N-2j-2k-1}\\
    &&=\S^N \sum_{j\geq 1}\left(\sum_{\ell=1}^j c_\ell z^{2\ell}(z/2)^{-2\ell} \frac{(2j+1)!}{(2j-2\ell+1)!} \right) \binom{N}{2j+1} z^{2j}\C^{N-(2j+1)}
\end{eqnarray*}
Therefore the lemma follows from the identity
$$
\frac{2j}{(2j+1)!}=[z^{2j}] \S(2z)\left(\C(2z)-1\right)= \sum_{\ell=1}^j \frac{c_\ell2^{2\ell}}{(2j-2\ell+1)!} .
$$
\end{proof}

\begin{proof}[{\red End of the proof of Lemma~\ref{lem:Ahyperbolic}}]
We set $$
X\coloneqq (-a_1)^{2g-4+n} {\rm exp}\left(a_1\frac{z\S'}{\S}\right) \frac{\prod_{i=2}^n \S(a_i)}{\S^{2g-1+n}}.
$$
With this notation we have
\begin{eqnarray}
\A &=& [z^{2g}]-a_1X, \text{ \red so } \\
    \frac{\partial \A}{\partial a_2}&=& [z^{2g}] \left(2g-3+n+a_1z\frac{\cS'}{\cS} - a_1z\frac{\cS'(a_2)}{\cS(a2)} \right) X \\
    &=& {\red [z^{2g}] \left(a_2+a1-1+\left(\sum_{i>2}a_i\right) +a_1z\frac{\cS'}{\cS} - a_1z\frac{\cS'(a_2)}{\cS(a2)} \right) X }\\
    a_2\A^2 + \A &=& [z^{2g-2}]\, \frac{X\S}{\S(a_2)}\, \int_{b=0}^{a_2-1} b(a_2-1-b)\S(b)\S(a_2-1-b) \\&& +\,  [z^{2g}]\,  \frac{X\S}{\S(a_2)}\, \sum_{i>2} (a_2+a_i-1) \frac{\S(a_2+a_i-1)\S}{\S(a_i)}. \nonumber
\end{eqnarray}
To simplify the last formula, we compute the integral 
\begin{eqnarray*}
2 {\int_{b=0}^y} \sinh(bz/2)\sinh((y-b)z/2)\, db &=& \int_{b=0}^y \left(\cosh(yz/2) + \cosh((b-y/2)z)\right) \, db\\
&=& y\cosh(yz/2) + \frac{2}{z} \sinh(yz/2) = y^2z \S'(y).   
\end{eqnarray*}
Then, we get 
\begin{eqnarray}\label{for:A2expression}
    a_2\A^2\!\!\!\! &+& \!\!\!\! \A \, =\,  [z^{2g}]\left((a_2-1)^2z\frac{\S'(a_2-1)}{\S(a_2)} + \sum_{i>2} (a_2+a_i-1) \frac{\S(a_2+a_i-1)}{\S(a_2)\S(a_i)} \right) X\S.
\end{eqnarray}
We simplify further this expression by setting
\begin{eqnarray*}
F\coloneqq X \S \frac{(a_2-1)\S(a_2-1)}{\S(a_2)}.
\end{eqnarray*}
The log-derivative of $F$ is equal to 
\begin{eqnarray*}\nonumber
\frac{1}{F}\frac{\partial F}{\partial z} &=&  a_1\left(\frac{\S'}{\S}+z\frac{\S''\S-\S'^2}{\S^2}\right) - (2g-2+n)\frac{\S'}{\S}  + \frac{(a_2-1)\S'(a_2-1)}{\S(a_2-1)}+ \sum_{i>2} \frac{a_i\S'(a_i)}{\S(a_i)}.
\end{eqnarray*}
We have $[z^{2g}] \left(z\frac{\partial F}{\partial z}-2gF\right)=0$. This identity can be rewritten as 
\begin{eqnarray}\label{for:Frelation}
    0&=& [z^{2g}] \left( a_1\left(\C-\frac{1}{\S^2}\right) - (2g-2+n)\C  + \frac{(a_2-1)z\S'(a_2-1)}{\S(a_2-1)}+ \sum_{i>2} \C(a_i) \right) F\\\nonumber
    &=& [z^{2g}] \left( \frac{-a_1}{\S^2} + \frac{(a_2-1)z\S'(a_2-1)}{\S(a_2-1)}-a_2\C + \sum_{i>2} (\C(a_i)-a_i\C) \right) F,
\end{eqnarray}
here we have used the fact that $|a|=2g-2+n,$ and $\S''=\frac{1}{4}\S-\frac{2}{z}\S'$
 to simplify the expression of the derivative of $F$. Besides, for all $i>2$, we have
\begin{eqnarray*}
    &&\frac{(a_2+a_i-1)\S(a_2+a_i-1)}{\S(a_2)\S(a_i)\S}- (\C(a_i)-a_i\C)\frac{(a_2-1)\S(a_2-1)}{\S(a_2)\S}\\
    &=& a_2\C(a_i)\C+a_i\C(a_2)\C-\C(a_i)\C(a_2)-\frac{a_2a_iz^2}{4}-(\C(a_i)-a_i\C)(a_2\C-\C(a_2))\\
    &=& {\red a_2a_i \C^2-\frac{a_2a_iz^2}{4}}=\frac{a_2a_i}{\S^2}.
\end{eqnarray*}
{\red From the first to the second line we have used the product formula~\eqref{for:hyperbolicproduct} and the last equality is obtained by applying the relation $\C^2=\frac{1}{\S^2}+\frac{z^2}{4}$}. We use these expressions to simplify the difference between~\eqref{for:A2expression} and~\eqref{for:Frelation} 
 \begin{eqnarray*}
     a_2\A^2+ \A &=& [z^{2g}]\left(\left(a_2\S^2\C+{a_1}\right)\frac{(a_2-1)\S(a_2-1)}{\S(a_2)\S} +\sum_{i>2} a_2a_i\right) X\\
     &=&{\red [z^{2g}]\left(\left(a_2\S^2\C+a_1\right)\left(a_2\C-\C(a_2)\right) +\sum_{i>2} a_2a_i\right) X}
 \end{eqnarray*}
 (we used~\eqref{for:hyperbolicproduct2} to obtain the second expression).  Therefore, using the expressions of $\A$ and its derivative, we get the expression:
\begin{eqnarray*}
    \frac{\partial \A}{\partial a_2}-\A^2&=& \frac{1}{a_2} [z^{2g}]\Bigg(a_2(a_1+a_2-1)-a_1+a_1a_2\frac{z\S'}{S}-a_1a_2\frac{z\S'(a_2)}{\S(a_2)}\\
    &&\,\,\,\,\,\,\,\,\,\,\,\,\,\,\,\,\,\,\,\, -
    \left(a_2\S^2\C+a_1\right)\left(a_2\C-\C(a_2)\right)  \Bigg) X\\
   &{\red =}& {\red \frac{1}{a_2} [z^{2g}]\Bigg(a_1\underset{=0}{\underbrace{\left(\C(a_2)-1-a_2\frac{z\S'(a_2)}{\S(a_2)}\right)}}+a_1a_2\underset{=0}{\underbrace{\left(1-\C+\frac{z\S'}{\S}\right)}}} \\
    &&\,\,\,\,\,\,\,\,\,\,\,\,\,\,\,\,\,\,\,{\red + a_2\left(a_2-1-a_2\S^2\C^2  + \C(a_2)\S^2\C \right) \Bigg) X} \\
    &=& [z^{2g}](a_2-1)\left(1-\frac{\S(a_2-1)}{\S(a)} \C {\red \S} \right)X,
\end{eqnarray*}
{\red where the last equality is obtained from~\eqref{for:hyperbolicproduct}}. If $a_2=N$, then we use Lemma~\ref{lem:hyperbolic} to simplify this last formula:
\begin{eqnarray*}
    \left(\frac{\partial \A}{\partial a_2}-\A^2\right)\bigg|_{a_2=N}&=& [z^{2g}]\left(\sum_{j\geq 1} c_{j,N}\,\frac{\S(N-2j)}{\S(N)}z^{2j}S^{2j}\right)X\\
    &=& \sum_{j\geq 1} c_{j,N}\,  \A(a_1,a_N-2j,\ldots).
\end{eqnarray*}
\end{proof}


\begin{thebibliography}{BCGGM19b}

\bibitem[BCGGM19a]{BCGGM3}
M.~Bainbridge, D.~Chen, Q.~Gendron, S.~Grushevsky, and M.~M\"{o}ller.
\newblock The moduli space of multi-scale differentials.
\newblock {\em arXiv:1910.13492}, 2019.

\bibitem[BCGGM19b]{BCGGM2}
M.~Bainbridge, D.~Chen, Q.~Gendron, S.~Grushevsky, and M.~M\"{o}ller.
\newblock Strata of {$k$}-differentials.
\newblock {\em Algebr. Geom.}, 6(2):196--233, 2019.

\bibitem[BHPPS23]{BHPSS}
Y.  Bae, D.  Holmes, R.  Pandharipande, J.  Schmitt, and R. 
  Schwarz.
\newblock Pixton's formula and {Abel}-{Jacobi} theory on the {Picard} stack.
\newblock {\em Acta Math.}, 230(2):205--319, 2023.

\bibitem[CLR00]{ComLioRol}
G.~Comte, J.-M. Lion, and J.-P. Rolin.
\newblock Log-analytic nature of the volume of subanalytic sets.
\newblock {\em Ill. J. Math.}, 44(4):884--888, 2000.

\bibitem[CMS23]{CheMoeSau}
D.~Chen, M.~M{\"o}ller, and A.~Sauvaget.
\newblock Masur--{Veech} volumes and intersection theory: the principal strata
  of quadratic differentials.
\newblock {\em Duke Math. J.}, 172(9):1735--1779, 2023.

\bibitem[CMS25]{CosMolSch}
Matteo Costantini, Martin M{\"o}ller, and Johannes Schwab.
\newblock Chern classes of linear submanifolds with application to spaces of
  {{\(k\)}}-differentials and ball quotients.
\newblock {\em Comment. Math. Helv.}, 100(4):745--809, 2025.

\bibitem[CMZ24]{CosMoeZac}
Matteo Costantini, Martin M{\"o}ller, and Jonathan Zachhuber.
\newblock The area is a good enough metric.
\newblock {\em Ann. Inst. Fourier}, 74(3):1017--1059, 2024.

\bibitem[CSS21]{CosSauSch}
M.~Costantini, A.~Sauvaget, and J.~Schmitt.
\newblock Integrals of psi-classes over higher double ramification cycles.
\newblock {\em arXiv:2112.04238}, 2021.

\bibitem[DN09]{DoNor}
N.~Do and P.~Norbury.
\newblock Weil-{Petersson} volumes and cone surfaces.
\newblock {\em Geom. Dedicata}, 141:93--107, 2009.

\bibitem[DSvZ21]{DelSchZel}
V.~Delecroix, J.~Schmitt, and J.~van Zelm.
\newblock admcycles -- a {Sage} package for calculations in the tautological
  ring of the moduli space of stable curves.
\newblock {\em J. Softw. Algebra Geom.}, 11:89--112, 2021.

\bibitem[FP18]{FarPan}
G.~Farkas and R.~Pandharipande.
\newblock The moduli space of twisted canonical divisors.
\newblock {\em J. Inst. Math. Jussieu}, 17(3):615--672, 2018.

\bibitem[GP20]{GhaPir}
S.~Ghazouani and L.~Pirio.
\newblock Moduli spaces of flat {T}ori and elliptic hypergeometric functions.
\newblock {\em M\'{e}m. Soc. Math. Fr. (N.S.)}, (164):viii+183, 2020.

\bibitem[KN18]{KozNgu}
V.~Koziarz and D.-M. Nguyen.
\newblock Complex hyperbolic volume and intersection of boundary divisors in
  moduli spaces of pointed genus zero curves.
\newblock {\em Ann. Sci. \'{E}c. Norm. Sup\'{e}r. (4)}, 51(6):1549--1597, 2018.

\bibitem[LZ04]{LanZvo}
S.~K. Lando and A.~K. Zvonkin.
\newblock {\em Graphs on surfaces and their applications. {Appendix} by {Don}
  {B}. {Zagier}}, volume 141 of {\em Encycl. Math. Sci.}
\newblock Berlin: Springer, 2004.

\bibitem[Mas82]{Mas}
H.~Masur.
\newblock Interval exchange transformations and measured foliations.
\newblock {\em Ann. of Math. (2)}, 115(1):169--200, 1982.

\bibitem[McM17]{McM}
C.~T. McMullen.
\newblock The {G}auss-{B}onnet theorem for cone manifolds and volumes of moduli
  spaces.
\newblock {\em Amer. J. Math.}, 139(1):261--291, 2017.

\bibitem[Mir07]{Mir1}
M.~Mirzakhani.
\newblock Simple geodesics and {W}eil-{P}etersson volumes of moduli spaces of
  bordered {R}iemann surfaces.
\newblock {\em Invent. Math.}, 167(1):179--222, 2007.

\bibitem[Ngu22]{Ngu}
D.-M. Nguyen.
\newblock Volume forms on moduli spaces of {{\(d\)}}-differentials.
\newblock {\em Geom. Topol.}, 26(7):3173--3220, 2022.

\bibitem[PZ]{PixZag}
A.~Pixton and D.~Zagier,
\newblock in preparation.

\bibitem[Sau18]{Sau4}
A.~Sauvaget.
\newblock Volumes and {S}iegel-{V}eech constants of $\mathcal{H}(2g-2)$ and
  {H}odge integrals.
\newblock {\em Geom. Funct. Anal.}, 28(6):1756--1779, 2018.

\bibitem[Sau19]{Sau}
A.~Sauvaget.
\newblock Cohomology classes of strata of differentials.
\newblock {\em Geom. Topol.}, 23(3):1085--1171, 2019.

\bibitem[Sch18]{Sch}
J.~Schmitt.
\newblock Dimension theory of the moduli space of twisted {$k$}-differentials.
\newblock {\em Doc. Math.}, 23:871--894, 2018.

\bibitem[Spe24]{Spe}
Pim Spelier.
\newblock Polynomiality of the double ramification cycle.
\newblock 2024, arXiv:2401.17421.

\bibitem[Thu98]{Thu}
W.~P. Thurston.
\newblock Shapes of polyhedra and triangulations of the sphere.
\newblock In {\em The {E}pstein birthday schrift}, volume~1 of {\em Geom.
  Topol. Monogr.}, pages 511--549. Geom. Topol. Publ., Coventry, 1998.

\bibitem[Tro86]{Tro}
M.~Troyanov.
\newblock Les surfaces euclidiennes \`a singularit\'{e}s coniques.
\newblock {\em Enseign. Math. (2)}, 32(1-2):79--94, 1986.

\bibitem[TWZ06]{TanWonZha}
S.~P. Tan, Y.~L. Wong, and Y.~Zhang.
\newblock Generalizations of {M}c{S}hane's identity to hyperbolic
  cone-surfaces.
\newblock {\em J. Differential Geom.}, 72(1):73--112, 2006.

\bibitem[Vee82]{Vee}
W.~A. Veech.
\newblock Gauss measures for transformations on the space of interval exchange
  maps.
\newblock {\em Ann. of Math. (2)}, 115(1):201--242, 1982.

\bibitem[Vee93]{Vee1}
W.~A. Veech.
\newblock Flat surfaces.
\newblock {\em Amer. J. Math.}, 115(3):589--689, 1993.

\end{thebibliography}

\newcommand{\etalchar}[1]{$^{#1}$}

\end{document}